\documentclass[11pt,a4paper]{amsart}

\usepackage[margin=3cm]{geometry}
\usepackage{amsmath,amsfonts,amssymb,amsthm}
\usepackage[dvipsnames]{xcolor}
\usepackage{graphicx}
\usepackage{enumitem}
\usepackage{calc}
\usepackage{url}

\usepackage[foot]{amsaddr}
\usepackage{hyperref}

\usepackage[normalem]{ulem}

\usepackage[linesnumbered,ruled,vlined]{algorithm2e}
\DontPrintSemicolon

\SetKwComment{Comment}{\color{green!50!black}// }{}

\newcommand{\FuncCall}[2]{\texttt{\bfseries #1(#2)}}
\SetKwProg{Function}{function}{}{}
\SetKwProg{Procedure}{procedure}{}{}

\newtheorem{theorem}{Theorem}
\newtheorem{lemma}{Lemma}

\newcommand{\pdpd}[2]{\frac{\partial #1}{\partial #2}}
\renewcommand{\Re}{\operatorname{Re}}

\newcommand{\RR}{{\mathbb{R}}}

\newcommand{\TT}{{\mathbb{T}}}
\newcommand{\ZZ}{{\mathbb{Z}}}

\theoremstyle{definition}
\newtheorem{definition}{Definition}
\theoremstyle{remark}

\title[A two-grid method for Helmholtz problems]%
{A two-grid method with dispersion matching for finite-element Helmholtz
  problems}
\author[C.C.\ Stolk]{Christiaan C.\ Stolk}
\address{University of Amsterdam, Korteweg-de Vries Institute, POBox
  94248, 1090 GE Amsterdam, The Netherlands}
\email{C.C.Stolk@uva.nl}

\begin{document}
\newcommand{\NewInRevision}[1]{{#1}}

\begin{abstract}
  This work is about a new two-level solver for Helmholtz equations
  discretized by finite elements. The method is inspired by two-grid methods
  for finite-difference Helmholtz problems as well as by previous work on
  two-level domain-decomposition methods. For the coarse-level
  discretization, a compact-stencil finite-difference method is used
  that minimizes dispersion errors. The smoother involves a
  domain-decomposition solver applied to a complex-shifted Helmholtz
  operator.  Local Fourier analysis shows the method is convergent if
  the number of degrees of freedom per wavelength is larger than some
  lower bound that depends on the order, e.g.\ more than 8 for order
  4. In numerical tests, with problem sizes up to 80 wavelenghts,
  convergence was fast, and almost independent of problem size unlike
  what is observed for conventional methods.
 Analysis and comparison with dispersion-error
  data shows that, for good convergence of a two-grid method for
  Helmholtz problems, it is essential that fine- and coarse-level
  dispersion relations closely match.
\end{abstract}

\maketitle

\section{Introduction}

In this paper we study the numerical solution of high-frequency
Helmholtz problems discretized by finite elements. Such problems, in
3-D, are still highly challenging.  The same is true for other
high-frequency time-harmonic wave problems such as those for elastic
and electromagnetic fields.

Extensive research has been done in this area, and there is wide range
of methods, typically with different pros and cons. We will say a few
things here.  For direct methods, the $N^{4/3}$ scaling leads to very
large memory requirements for medium to large problems, which
has motivated extensive research into iterative methods.  For iterative
methods different types of behavior have been observed. Some
methods, like some domain decomposition methods, feature expensive (or
difficult to parallellize) steps, but require relatively few of those
\cite{engquist2011sweeping, stolk2013rapidly, vion2014double,gander2019class}. For
others, including shifted-Laplacian methods or time-domain methods, the cost per
iteration is relatively small, but many of those are required
\cite{erlangga2006novel, CalandraEtAl2013, grote2019controllability,
  appelo2020waveholtz, stolk2021time}. In this
case there is often a notable difference between finite-element and
finite-difference discretizations. For finite differences, the steps
in the iteration can be much cheaper due
to the high sparsity and regular matrix structure, especially when GPUs
are used \cite{knibbe20133d,TurkelEtAl2013,stolk2021time,CalandraEtAl2013}.
As already mentioned we consider finite-element discretizations.

The specific topic of this work is a new two-grid approach. The costly
steps in such algorithms are smoothing steps, and coarse-level
solves. There are a couple of considerations which will determine the
design of the algorithm:
\begin{itemize}
\item
  As usual, the smoother should perform local computations in the sense of a
  Jacobi or a domain decomposition method. These are easy to
  parallellize.
\item
  In the coarse-level problem, the number of
  degrees of freedom is reduced by a factor $2^d$, where $d$ is the
  dimension, compared to the original problem. So a factor 8 in 3-D.
  The coarse problem is solved by some solver (here a direct solver,
  but in principle, approximate solvers may be used
  \cite{stolk2017improved}). Depending on this
  solver, the cost in terms of memory or computation time can be more
  than $2^d$ smaller compared to the original problem.
  While being much smaller than the original problem, the coarse-level
  problem remains quite large, so the number of iterations should be
  relatively small.
\item
  For good two-grid convergence we believe that the coarse-level
  solutions should well reproduce the propagating waves of the
  original solutions. In particular waves should not be too strongly
  damped and coarse-level wavelengths should be close to the correct
  wavelengths, i.e.\ dispersion errors should be small. Indeed, having
  small coarse-level dispersion errors, really makes the difference
  between good and poor convergence, as was argued in previous works
  for finite-difference discretizations \cite{StolkEtAl2014,
    stolk2017improved}.  One may wonder whether this requirement is
  not too strict. However, finite-difference discretizations exist
  with very small dispersion errors on meshes using as little as three
  degrees of freedom per wavelength
  \cite{BabuskaEtAl1995,StolkEtAl2014,stolk2016dispersion}.
\end{itemize}

The last bullet point motivates this paper:
we propose a two-grid method using a finite-difference method with
very small dispersion errors at the coarse level. In four words this
might be called a {\em dispersion matching two-grid method}.
Of course there are also other components to a two-grid method.
For the smoother we propose an overlapping domain decomposition
method applied to a complex-shifted Helmholtz problem, somewhat
comparable to methods discussed in \cite{graham2017domain},
cf.\ \cite{kimn2013shifted}.
The further description of the method is in section~\ref{sec:two-grid-method}
while subsection \ref{subsec:QSFEM_coarse} contains details about
how the finite-difference coarse-level discretization is
embedded in a finite-element two-grid method.

The main results of the paper are as follows: We present a new
two-grid method along the above lines for 2-D finite-element Helmholtz
problems, involving various discretization orders and boundary
conditions, see section~\ref{sec:two-grid-method}.  Convergence
factors of the method are estimated using local Fourier analysis and
compared to convergence factors for standard Galerkin coarsening, cf.\
Figures~\ref{fig:rho_lfa_fd_vs_pcoarsen}
and~\ref{fig:rho_lfa_mgpar}.
The convergence factors depend on the
number of dofs per wavelength and (i) are much better than convergence
factors for standard Galerkin coarsening, and (ii) are good basically
whenever the fine mesh is fine enough to guarantee small dispersion
errors for the original problem%
\footnote{%
\NewInRevision{%
Convergence deteriorates due to wavelength differences between
coarse and fine level discretizations. Fine level
dispersion errors may contribute to these, which is related to the
fact that in this work the coarse level
discretization is chosen to minimize errors with the true dispersion
relation, not with the fine level dispersion relation.}}.
An approximate relation between
asymptotic convergence rate and dispersion errors is described in
sections~\ref{sec:lfa_toy} and~\ref{sec:lfa_full_method}.  Numerical
experiments confirm good convergence for finite elements of different
orders and all domain sizes that where studied (up to 80 wavelengths),
see section~\ref{sec:numerical_results}.

An important restriction is that we use regular meshes. The reason is
that this allows to do local Fourier analysis. Also it is not
immediately clear how to generalize the optimized finite differences
of \cite{BabuskaEtAl1995,StolkEtAl2014,stolk2016dispersion}
to general meshes. Also we do not pursue here the
generalization to 3-D. While the formulation of the method in 3-D is
straightforward, computations in 3-D are much more demanding
in terms of implementation and hardware and the potential of the
method is already clearly illustrated by the 2-D results.

It can be observed that there is a certain similarity of our method with the
method of
\cite{graham2017domain}: This paper also describes a two level method,
with overlapping domain decomposition at the fine level.  Differences
with our work are that we do different coarsening, using a factor of 2
in each direction, that we keep at least three degrees of freedom per
wavelength in the coarse mesh, and that we use schemes with very small
dispersion errors at the coarse level. This leads to better
convergence. Also our numerical results are much more extensive,
including in particular higher order elements. On the other hand
\cite{graham2017domain} rigorously proves some convergence results,
while we use local Fourier analysis.
Other related work includes work on dispersion and pollution errors in finite
elements \cite{ihlenburg1995dispersion, babuska1997pollution}, on
coarse-level discretization in time-harmonic elasticity
\cite{yovel2024lfa}, and on dispersion correction
\cite{cocquet2021closed, cocquet2024asymptotic}.

The rest of the paper is organized as follows. In
section~\ref{sec:continuous_Helmholtz} we formulate the Helmholtz
problem. Section~\ref{sec:discretizations} covers the relevant
discretizations. The new two-grid method is then described in
section~\ref{sec:two-grid-method}. Local Fourier analysis is applied
to a 1-D toy problem in section~\ref{sec:lfa_toy} and to the full
problem in section~\ref{sec:lfa_full_method}.
Section~\ref{sec:numerical_results} contains numerical
results. Finally section~\ref{sec:discussion_conclusions} contains a
discussion and some conclusions.

\section{Continuous Helmholtz problem\label{sec:continuous_Helmholtz}}

The following Helmholtz problem is considered:
 \begin{equation} \label{eq:Helmholtz_eq}
   - \Delta u - \left( k(x)^2 + i \epsilon(x)\right) u = f(x)
\qquad\qquad \text{in $\Omega$} ,
\end{equation}
with boundary conditions of Dirichlet, Neumann or absorbing (Robin)
type. It is
assumed that there are positive constants $C_1, C_2, C_3$ such that
$C_1 \le k(x) \le C_2$ and $0 \le \epsilon(x) \le C_3$. Because of the
square-mesh assumption, see the introduction, the domain is assumed to
be a union of squares. The boundary is assumed to be a union
$\partial \Omega = \Gamma_0 \cup \Gamma_1 \cup \Gamma_{\rm abs}$
where $\Gamma_0, \Gamma_1, \Gamma_{\rm abs}$ overlap only in
codimension 2 sets, and boundary conditions are of Dirichlet type on
$\Gamma_0$, of Neumann type on $\Gamma_1$ and of absorbing type on
$\Gamma_{\rm abs}$, so
\begin{equation}
  \begin{aligned}
    u = {}& 0 \qquad x \in \Gamma_0 ,
    \\
    \pdpd{u}{n} = {}& 0 \qquad x \in \Gamma_1 ,
    \\
    \pdpd{u}{n} - i k u = {}& 0 , \qquad x \in \Gamma_{\rm abs} .
  \end{aligned}
\end{equation}
In the numerical examples we always assume that absorption is present
in the problem at least on part of the boundary, either via absorbing
boundary conditions, so non-empty $\Gamma_{\rm abs}$, or via absorbing
layers, i.e.\ strictly positive $\epsilon(x)$ on some layer at a
boundary.

The associated bilinear form for finite-element discretization is
\begin{equation} \label{eq:bilin_form}
  \begin{aligned}
  a( u,v)
  = {}& a(u,v; \Omega, \Gamma_{\rm abs})
  \\
  = {}& \int_\Omega \nabla u(x) \cdot \nabla v(x) \, dx
  - \int_\Omega (k(x)^2+i\epsilon) u v \, dx
  - \int_{\Gamma_{\rm abs}}  i k u v \, ds .
\end{aligned}
\end{equation}
The notation $a(u,v; \Omega, \Gamma_{\rm abs})$, explicitly indicating
the domains of the integrals, will be used below to define
bilinear forms on subdomains.
As usual, the Dirichlet boundary condition is enforced by requiring
solutions to be in 
$H_0^1(\Omega; \Gamma_{0}) = \{ u \in H^1(\Omega) \, : \,
u(x) = 0 \text{ for } x \in \Gamma_0 \}$.

\section{Discretizations\label{sec:discretizations}}

This work involves two types of discretizations, one on the fine mesh
and one on the coarse mesh of the two-grid method. The linear system
to be solved is a Helmholtz problem discretized by a continuous
Galerkin method on the fine mesh. Internally in the solver, as a
coarse-level discretization, there is an optimized finite-difference
discretization on the coarse mesh. Both are discussed in this section.

\subsection{Finite-element discretization}

Given the bilinear form in (\ref{eq:bilin_form}) one needs to specify
\NewInRevision{a mesh, a finite-element space and finite element basis functions
to describe the discretization}.  It was
already indicated that regular meshes are used. In fact, 
it is assumed that the domain $\Omega$ is a union
of cells of a square mesh, which will be called the `underlying'
mesh. This mesh is either used directly (square mesh) or the
square cells are split into triangles (triangular mesh).
\NewInRevision{Examples are given in Figure~\ref{fig:simple_meshes}.}
\begin{figure}
  \begin{center}
    \includegraphics[width=3cm]{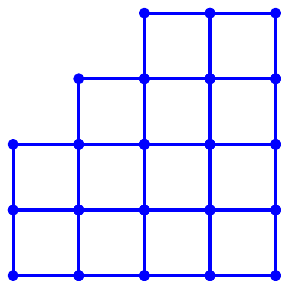}
\hspace*{1cm}
    \includegraphics[width=3cm]{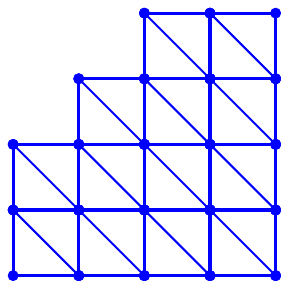}
  \end{center}
  \caption{Examples of a square and a triangular mesh of the type
    considered.}\label{fig:simple_meshes}
\end{figure}
In either case the spacing of the underlying mesh is denoted $h$.
\NewInRevision{Standard continuous Galerkin finite element spaces of
  order $p$ are used with $p=2,4,6$ or $8$.
  On triangular elements, the set of polynomials is hence
  $\mathcal{P}_p$, while on the square mesh mesh cells,
  the set of polynomial is $\mathcal{Q}_p$
  (in the notation of \cite{brenner2008mathematical}).
  Finite element computations for this paper were done using the
software package NGSolve \cite{ngsolve}. The H1 finite element space of
NGSolve was used, which is a continuous Galerkin space using
hierarchical element basis functions.}


\subsection{Finite-difference discretizations with minimal wavelength
  errors: Quasi-stabilized FEM\label{subsec:fd_discretization}}

In \cite{BabuskaEtAl1995}, a 2-D finite-difference discretization with
minimal wavelength errors was introduced called quasi-stabilized FEM
(QSFEM). This discretization, with a multiplicative scaling, will be
used as coarse-level discretization. In 3-D, a method with similar
properties was introduced in \cite{stolk2016dispersion}.  We will
briefly explain the special properties of QSFEM.

In case $k$ is constant, $\epsilon = f = 0$ and $\Omega = \RR^d$, the
Helmholtz equation has `propagating' plane wave solutions
$e^{i x \cdot \xi}$, where $\xi$ satisfies $\| \xi \| = k$. The
wavelength of these solutions is of course
$\lambda = \frac{2 \pi}{k}$. For good two-grid convergence it is
important that the coarse-level discretization has propagating wave
solutions with wavelength as close as possible to the true wavelength
(actually to the wavelength of the fine level discretization, but this
in turn should be close to the true wavelength) \cite{StolkEtAl2014}.

When a (constant coefficient) finite-difference operator acts on a
plane wave, the latter is multiplied by the symbol. Indeed, consider a
2-D grid $x_\alpha = h \alpha$, $\alpha \in \ZZ^2$, with $u_\alpha$
denoting a function on such a grid. 
Let $P$ be a finite-difference operator on such functions
\begin{equation}
  (P u)_\alpha = \sum_{\gamma \in S} p_\gamma u_{\alpha + \gamma} ,
\end{equation}
where $S$ is the set of stencil coefficients.
If such an operator is applied to a plane wave $e^{i x \cdot \xi}$, the
result is $\sigma_P(\xi) e^{i x \cdot \xi}$, where the multiplicative
factor $\sigma_P(\xi)$ is called the symbol and given by
\begin{equation} \label{eq:symbol_sigma_P}
  \sigma_P(\xi) = \sum_{\gamma \in S} p_\gamma e^{i h \gamma \cdot \xi} .
\end{equation}
If $P$ is a discretization of $-\Delta - k^2$, the propagating wave
solutions are those plane waves $u_\alpha = e^{i \xi \cdot x_\alpha}$
with $\xi$ such that $\sigma_P(\xi) = 0$.

Let $Z_P$ denote the zeroset of $\sigma_P$.
It was observed in \cite{BabuskaEtAl1995} that, to minimize pollution
errors, a notion of distance of this zero-set to the true zero-set $\{ \xi
\,:\, \| \xi \| = k \}$, defined by
\begin{equation}
  \mathcal{D}_P
  :=
  \max_{t \in [-\pi,\pi]}
  \min_{\xi \in Z_P}
  \left\| k \begin{pmatrix} \cos t\\ \sin t \end{pmatrix} - \xi \right\|
\end{equation}
should be minimal. 
Furthermore, for the 2-D case, the authors
constructed a finite-difference operator
with a compact (i.e.\ $3 \times 3$) stencil
that satisfied this criterion, and is given (in stencil notation) by
\begin{equation} \label{eq:P_QSFEM}
  P = \begin{bmatrix}
    P_2 & P_1 & P_2 \\
    P_1 & P_0 & P_1 \\
    P_2 & P_1 & P_2 
  \end{bmatrix}
\end{equation}
with
\begin{equation}
  \begin{aligned}
    P_0 = {}& \tilde{q}_0(\eta)  := 4
    \\
    P_1 = {}& \tilde{q}_1(\eta)  := 2 \frac{c_1(\eta) s_1(\eta) - c_2(\eta) s_2(\eta)}
    {c_2(\eta) s_2(\eta) (c_1(\eta)+s_1(\eta))
      - c_1(\eta) s_1(\eta) (c_2(\eta)+s_2(\eta))}
    \\
    P_2 = {}& \tilde{q}_2(\eta)  :=\frac{c_2(\eta) + s_2(\eta) - c_1(\eta) - s_1(\eta)}
    {c_2(\eta) s_2(\eta) (c_1(\eta)+s_1(\eta))
      - c_1(\eta) s_1(\eta) (c_2(\eta)+s_2(\eta))}
  \end{aligned}
\end{equation}
where $\eta = k h$, 
$c_1(\eta) = \cos \left( \eta \cos \frac{\pi}{16} \right)$, 
$s_1(\eta) = \cos \left( \eta \sin \frac{\pi}{16} \right)$, 
$c_2(\eta) = \cos \left( \eta \cos \frac{3\pi}{16} \right)$, and
$s_2(\eta) = \cos \left( \eta \sin \frac{3\pi}{16} \right)$.  
This is the QSFEM discretization referred to above.

Note that the stencil still needs to be scaled in
order to be a discretization of the Helmholtz equation.
The correctly scaled operator should satisfy
\NewInRevision{$\sigma_P(0) = - k^2$}. In this case
the symbol equals
\begin{equation}
  \sigma_P(\xi) = P_0 + 2 P_1 (\cos (h\xi_1) + \cos (h\xi_2))
  + 4 P_2 \cos (h\xi_1) \cos (h\xi_2) 
\end{equation}
and it is easy to check numerically
that \NewInRevision{$\sigma_P(0) \neq - k^2$}. Therefore, we scale the
operator by
\begin{equation} \label{eq:scaling_factor_QSFEM}
  h^{-2} N(hk) , \qquad
  \NewInRevision{ N(\eta) = \frac{- \eta^2}{\sigma_P(0)} . }
\end{equation}

\subsection{Coarse-mesh discretization using QSFEM\label{subsec:QSFEM_coarse}}

As explained in the introduction we propose to use QSFEM as a
coarse-level discretization for a finite-element two-grid solver.
However, in the finite-element two-grid algorithm, the coarse-level
discretization is itself also a finite-element discretization and not
a finite-difference discretization like QSFEM. The question is
therefore whether or how QSFEM can be incorporated in this algorithm.

To address this we exploit the fact that a first order finite-element
method is also a scaled finite-difference method. For example, in 2-D, the
matrix $A^{\rm FE,S}$, corresponding to the bilinear form
$\int \nabla u \cdot \nabla v \, dx$ is
\begin{equation}
  A^{\rm FE,S} = \begin{bmatrix} -\frac{1}{3} & -\frac{1}{3} & -\frac{1}{3}\\
    -\frac{1}{3} & \frac{8}{3} & -\frac{1}{3}\\
  -\frac{1}{3} & -\frac{1}{3} & -\frac{1}{3}\end{bmatrix} .
\end{equation}
This is $h^2$ times a finite-difference discretization of $-\Delta$.
The matrix $A^{\rm FE,M}$, corresponding to bilinear form
$\int u v \, dx$ is
\begin{equation}
  A^{\rm FE,M} = h^2 \begin{bmatrix} \frac{1}{36} & \frac{1}{9} & \frac{1}{36}\\
    \frac{1}{9} & \frac{4}{9} & \frac{1}{9}\\
  \frac{1}{36} & \frac{1}{9} & \frac{1}{36}\end{bmatrix} .
\end{equation}
which is $h^2$ times a second order discretization of the identity.
The full operator $-\Delta - (k^2 + i \epsilon)$ leads to the finite
element matrix (at least the internal matrix elements)
\begin{equation} \label{eq:fd_fe_constant1}
  A^{\rm FE,S} - (k^2 + i \epsilon)  A^{\rm FE,M}
\end{equation}
considering constant $k$ and $\epsilon$, which is hence $h^2$ times a
finite-difference discretization of the same operator.

For internal points and constant $k, \epsilon$, it is quite straightforward to 
obtain a QSFEM discretization with finite-element like scaling. In
(\ref{eq:fd_fe_constant1}) one 
replaces the part $A^{\rm FE,S} - k^2 A^{\rm FE,M}$
with $h^2$ times the QSFEM operator $h^{-2} N(hk )P$, hence by $N(hk)
P$ (see (\ref{eq:P_QSFEM}) and
(\ref{eq:scaling_factor_QSFEM})). 
\NewInRevision{Nothing is done with the reminaing
mass matrix $ i \epsilon  A^{\rm FE,M}$.}

The extension to variable $k$ and $\epsilon$ will be done in a way
that resembles the definition of finite-element matrices. Let $\alpha
\in \ZZ^d$ be an index for the degrees of freedom, which correspond to
the vertices of the mesh, denoted $v_\alpha$, and let $C(\alpha,\beta)$ denote the cells
of which both $v_\alpha$ and $v_\beta$ are vertices. In 2-D there can
be 4 elements (if $\alpha = \beta$), 2 elements (if $v_\alpha, v_\beta$ on opposite
sides of an edge), 1 element (if $v_\alpha,v_\beta$ on opposite
sides of a face) or 0 elements  (otherwise).
We will assume that $k$ and $\epsilon$ are constant over
cells, and denote these values by $k(c)$, $\epsilon(c)$, if $c$ is a cell.
The following functions will be invoked for getting elementwise matrix contributions 
\begin{equation}
  q_\gamma(\eta) = \left\{
    \begin{array}{ll}
      \frac{1}{4} N(\eta) P_0(\eta)
      & \text{if $\gamma = (0,0)$}
      \\
      \frac{1}{2} N(\eta) P_1(\eta)
      & \text{if $\gamma =(\pm 1,0)$ or $(0,\pm 1)$}
      \\
      N(\eta) P_2(\eta)
      & \text{if $\gamma = (\pm 1, \pm 1)$ or $(\pm 1, \mp 1)$.}
    \end{array}
  \right.
\end{equation}
The internal matrix elements $A_{\rm q}$ (the subscript referring to
QSFEM) are then given by
\begin{equation} \label{eq:element_wise_qsfem}
  (A_{\rm q})_{\alpha,\beta}
  = \sum_{c \in C(\alpha,\beta)} \left( q_{\beta-\alpha}(h k(c))
  - \int_c i \epsilon(c) \phi_\alpha(x) \phi_\beta(x) \, dx \right) .
\end{equation}
where $\phi_\alpha$ is the basis function for vertex $v_\alpha$.

Boundaries are treated as follows. First the sets $C(\alpha,\beta)$ are
defined to only contain cells that are part of the domain, and
secondly boundary terms like those in (\ref{eq:bilin_form}) are included
in the bilinear form.  Thus the following `quasi finite element'
version of the QSFEM discretization is obtained
\begin{equation}
  (A_{\rm q})_{\alpha,\beta}
  =
\sum_{c \in C(\alpha,\beta)} q_{\beta-\alpha}(h k(c))
  - \int_\Omega i\epsilon(x) \phi_\alpha(x) \phi_\beta(x) \, dx
  - \int_{\Gamma_{\rm abs}}  i k \phi_\alpha(x) \phi_\beta(x) \, ds .
\end{equation}
One can show that this leads to correct finite-difference boundary
conditions.

\section{The two-grid method\label{sec:two-grid-method}}

The two-grid methods proposed here have a conventional
structure, see algorithm~\ref{alg:two-grid_step}.
A coarse-grid correction is preceded and followed by
$n_{\rm s } \ge 1$ smoother steps. The coarse-grid correction also has a
conventional structure: The residual $r$ is mapped to an element
$r_{\rm c}$ of an appropriate coarse space by \NewInRevision{the restriction operator
which is} the adjoint
of a prolongation map $I_{\rm P}$; a discrete Helmholtz equation on
the coarse space is solved with right hand side $r_{\rm c}$; the
result is mapped back to the fine space by $I_{\rm P}$
and used to update the current approximation of the solution. 
The coarse-grid correction may be applied with a relaxation factor
$\omega_{\rm c}$.
Natural values for the parameters $n_{\rm s}$ and $\omega_{\rm c}$
are $n_{\rm s} = 1$ and
$\omega_{\rm c} = 1$, the effect of other values is studied in
sections~\ref{sec:lfa_full_method} and~\ref{sec:numerical_results}.

To describe the method, a couple of things are hence to be specified:
the smoother algorithm {\ttfamily CSDDSmoother}, the coarse
space, the prolongation map $I_{\rm P}$ and the coarse-level discrete
Helmholtz operator $A_{\rm c}$. This is done in the next subsections.

Note that algorithm~\ref{alg:two-grid_step}
updates a current approximate solution, like in a fixed-point
iteration. A left-preconditioner (map from $f$
to an approximate solution $u$) is obtained by setting $u =0$ as input
and taking the output value of $u$ as the result of the map.

\begin{algorithm}
  \caption{Two-grid step that updates $u$\label{alg:two-grid_step}}
  \Procedure{TwoGridStep($u, f$)}{
    \For{$i=1, \ldots, n_{\rm s}$}{
      \FuncCall{CSDDSmoother}{$u, f$} }
    $r_{\rm c} \gets I_{\rm P}^* (f - A u)$ \tcp*{Get coarse-grid residual}
    solve $u_{\rm c}$ from $A_{\rm c} u_{\rm c} = r_{\rm c}$ \tcp*{Get coarse-grid correction} 
$u \gets u + \omega_{\rm c} I_{\rm P} u_{\rm c}$ \tcp*{Apply coarse-grid correction}
    \For{$i=1, \ldots, n_{\rm s}$}{
      \FuncCall{CSDDSmoother}{$u, f$} }
  }
\end{algorithm}

\subsection{Smoother}

For finite-difference Helmholtz equations, after careful testing, a
few iterations of $\omega$-Jacobi, \NewInRevision{relaxing individual degrees of
freedom,} was found to be a good smoother \cite{StolkEtAl2014}.
\NewInRevision{For finite element equations in general,
  block relaxation methods, lead to better convergence. These are
  closely related to domain decomposition methods, and both are
  natural candidates for smoothers.}
Here a domain decomposition method will be used, however it will be
applied to a {\em complex-shifted Helmholtz equation}
\begin{equation}
  - \Delta u - ( k^2 +i \alpha_{\rm s} k^2 + i \epsilon ) ,
\end{equation}
and not to the original Helmholtz equation.

There are multiple reasons for the complex shift. Without it, domain
decomposition methods either converge poorly, or require many
sequential iterations that are difficult to parallellize.  Obviously
this is undesirable for a smoother.  On the other hand, for the large
wave numbers \NewInRevision{represented in the discretization},
the $-\Delta$ is the dominant contribution to the
Helmholtz operator $-\Delta - k^2$. That remains the case when a
complex shift is applied and a reduction in the large wave-number part
of the residual can therefore still be expected when a complex shift
is present. A complex shift has frequently been used before in
multilevel preconditioners for Helmholtz operators, see e.g.\
\cite{erlangga2006novel, kimn2013shifted, gander2015applying, graham2017domain},
in which case it is present in the both the smoother and the
coarse-level operators and not only in the smoother as here.

The first step in describing the smoother is the definition of a
complex-shifted Helmholtz operator. 
The matrix for this operator will
be denoted by $A_{\rm s}$. It is obtained by replacing the
coefficient $k^2+ i \epsilon$ in the bilinear form by
$k^2 + i \alpha_{\rm s} k^2 + i \epsilon$, 
where $\alpha_{\rm s}$ is a constant to be fixed later
(see sections~\ref{sec:lfa_full_method} and~\ref{sec:numerical_results}).

In principle, the smoother should update the approximate solution with
$A_{\rm s}^{-1}$ times the residual. However, since the inverse
$A_{\rm s}^{-1}$ is in general difficult to compute, it is
approximated by one or a few iterations of a domain decomposition
method.

For the domain decomposition, let
$\{ U^{(i)} \,:\, i = 1,\ldots, M_{\rm dd} \}$ be a non-overlapping
decomposition of the mesh, i.e. the overlaps are of codimension at
least one, and define subdomains $\Omega^{(i)}$ by adding one
layer of neighboring cells (squares of the underlying mesh) to the $U^{(i)}$.
Here neighboring means sharing an edge or vertex.
On the subdomains $\Omega^{(i)}$, finite-element Helmholtz matrices
$A^{(i)}$, are defined by using the bilinear form
$a(u,v; \Omega, \Gamma_{\rm abs})$ of (\ref{eq:bilin_form}) with
domain $\Omega^{(i)}$, using absorbing boundary conditions on the
internal boundaries and inherited boundary conditions from parts of
the boundary that were also part of the boundary of the original
domain.

A one-level additive-Schwarz method is now defined as follows. Let $u$
be the current approximate solution.
On each $\Omega^{(i)}$ an approximation of a solution $w^{(i)}$
on $\Omega^{(i)}$ can be obtained by solving
\[
  A^{(i)} w^{(i)} = (A^{(i)} R_i - R_i A ) u + R_i f .
\]
There are several existing methods to combine these into a  global
solution, see e.g.\ \cite{graham2017domain}, p.\ 2116. For
this work we use a variant where all dofs that are in only one of the $U^{(i)}$ are set by that
$w^{(i)}$, and all dofs that are in multiple $U^{(i)}$ are defined
to be the average. Note that this is different from both the AVE and the
RAS variants of \cite{graham2017domain}.
The new approximate solution is then (here $j$ is an index to the
dofs of the finite-element space)
\begin{equation} \label{eq:RAS_like_update}
  u^{({\rm next})}_j
  = \frac{1}{L_j} \sum_{i : j \in \text{Dofs}(U^{(i)} )} (R_i^T w^{(i)})_j 
\end{equation}
where $L_j$ denotes the number of subdomains $U^{(i)}$ such that
$j \in \text{Dofs}(U^{(i)} )$. Algorithm~\ref{alg:csdd} provides an
overview of the steps.

\begin{algorithm}
  \caption{Complex-shifted domain decomposition smoother\label{alg:csdd}}%
  \Procedure{CSDDSmoother($u, f$)}{
    $r \gets f - A u$ \;
    $v \gets 0$ \;
    \For{$i= 1, \ldots, n_{\rm dd}$}{
      \FuncCall{DDStep}{$v, r; A_{\rm s}$}  }
    $u \gets u + v$ }
  \Procedure{DDStep($u, f; A_{\rm s}$)}{
    \For {$i=1, \ldots, M_{\rm dd}$}{
      solve $w_i$ from $A_{\rm s}^{(i)} w_i = -(R_i A_{\rm s} - A_{\rm s}^{(i)} R_i) u + R_i f$}
    assign $u$ according to (\ref{eq:RAS_like_update})}
\end{algorithm}

\subsection{Coarse-grid correction using QSFEM\label{subsec:coarse-grid_correction}}

For the coarse-grid correction step of a finite-element two-grid
method, one needs a coarse finite-element space, a discretization on
this space
and a prolongation operator $I_{\rm P}$ that maps elements of the
coarse space to elements of the original, fine space. When QSFEM is
used as coarse-level discretization, the discretization is as described
in subsection~\ref{subsec:QSFEM_coarse}, and the choice of mesh is
straightforward. A square mesh is taken such that the number of
degrees of freedom per unit length is divided by two in each
direction. The fine mesh involves order $p$ elements on cells of size
$h$ (or on triangles obtained by splitting the squares), the coarse
mesh will therefore be square with spacing $\frac{2 h}{p}$.
The prolongation operator will be given by interpolation, i.e.\ if $u
\in V_{\rm c}$ we require that
\begin{equation}
  \begin{aligned}
  {}& \text{$I_{\rm P}  u$ is piecewise polynomial of degree $p/2$ on each fine
    mesh cell, and}
  \\
  {}& I_{\rm P} u (x_\alpha) = u(x_\alpha) \qquad \text{for all vertices $x_\alpha$
    of the coarse mesh} .
\end{aligned}
\end{equation}
\NewInRevision{%
  These prolongation operators are constructed numerically,
  following a standard
pattern in FEM interpolation, in which the vertex degrees of freedom
are determined first, then de edge degrees of freedom and then the
cell degrees of freedom.  Let $u_{\rm c}$ be some coarse function
and $u = P u_{\rm c}$ its prolongation. First the degrees of freedom
of $u$ for each fine mesh vertex are determined, from the values of
$u_{\rm c}$ at these vertices, which are also vertices of the coarse
mesh. Denote the intermediate result by $u^{({\rm V})}$.  Then the
degrees of freedom of $u$ for each edge are determined. The lowest
$p/2-1$ edge basis functions in the hierarchical basis have nonzero
coefficients. These are determined by interpolating the function
$u_{\rm c} - u^{({\rm V})}$ at the $p/2-1$ coarse mesh vertices that
lie on the interior of the edge. This gives a second intermediate
result for $u$, denoted $u^{({\rm VE})}$. Then the degrees of freedom
corresponding to the interior of each cell are determined.  E.g.\ for
the square mesh the lowest $(p/2-1)^2$ of the hierarchical face basis
functions have nonzero coefficients, which are determined by
interpolating $u_{\rm c} - u^{({\rm VE})}$ through the $(p/2-1)^2$
vertex points of the coarse mesh that lie in the interior of the
cell. Now all degrees of freedom of $u$ are determined.} This concludes
the specification of the coarse-grid correction in case the QSFEM
discretization is used.

\subsection{Standard Galerkin $p$-type coarsening}

For comparison, standard Galerkin $p$-type coarsening will also be
considered. In this case the coarse and fine spaces are based on the same mesh, but
the coarse space uses elements of lower order $p_{\rm c} = p/2$ compared to the fine space. The matrix
$A_{\rm c}$ follows directly from this and the prolongation map is
simply the identity map since $V_{\rm c} \subset V$. 

In case of Galerkin $p$-type coarsening, convergence often improved
when a small complex shift was included in the coarse-level discrete
Helmholtz operator, see the numerical results in
section~\ref{sec:numerical_results}.

\subsection{Parameters of the method\label{subsec:parameters_of_method}}

The method contains a number of parameters, which are listed here.
The parameters related to the domain decomposition are: the
size of the subdomains $l_{\rm dd}$, the subdomain overlap, the number of domain
decomposition iterations $n_{\rm dd}$. The subdomains $U^{(i)}$ will
be chosen as $d$-dimensional blocks with dimension $l_{\rm dd}$ in
each direction. Of course since they form a partition, some subdomains
may need to have slightly different size, i.e.\ one smaller. The
$\Omega^{(i)}$ are obtained by adding one layer of squares to the
$U^{(i)}$ so the overlap is 2 cells throughout.

Other parameters for the smoother are: The parameter $\alpha_{\rm s}$
for the complex shift, the number of pre- and postsmoother steps
$n_{\rm s}$ in a
multigrid cycle and the relaxation factor in the coarse-grid
correction $\omega_{\rm c}$. 

In case of Galerkin $p$-type coarsening, there is in addition a
parameter $\alpha_{\rm c}$ for the complex shift in the coarse-level operator.

\section{Local Fourier analysis: Relation between convergence and
  dispersion errors for a toy problem\label{sec:lfa_toy}}

In this section, local Fourier analysis is applied to a toy problem,
in casu a 1-D, finite-difference version of the above algorithm.  In
this case one can more or less read off the convergence factor from
simple analytical expressions involving the dispersion errors, the
number of points per wavelength and an artificially introduced damping
needed to keep the problem well-posed. Local Fourier analysis for
multigrid applied to finite-difference problems is described in
\cite{brandt1977multi, stuben1982multigrid,
  TrottenbergOosterleeSchueller2001}. We follow \cite{
  TrottenbergOosterleeSchueller2001}.

\subsection{Toy problem and two-grid method\label{subsec:damped_Helmholtz_for_LFA}}
For the local Fourier analysis, the following constant coefficient problem with
damping, i.e.\ nonzero $\epsilon$ is studied
\begin{equation}
  - \frac{d^2 u}{dx^2} - (k^2 + i \epsilon) u = f(x), \qquad x \in \RR .
\end{equation}
Damping is required, because the operator $-\frac{d^2 u}{dx^2} - k^2$
is not invertible. Since, in practice, one is interested in waves
propagating over large but finite distances, a small
amount of damping is not a problem.
The damping constant $\epsilon$ is given in terms of a dimensionless
constant $D$ by
\begin{equation} \label{eq:epsilon_ksquare_D}
  \epsilon = \frac{k^2 D}{\pi} .
\end{equation}
The fundamental solution of this equation is given by
\begin{equation}
  \frac{i e^{i \kappa |x|}}{2 \kappa} 
  \qquad \text{with} \qquad
  \kappa  = \sqrt{k^2 + i \epsilon}.
\end{equation}
Since $\kappa = k \sqrt{1 + \frac{i D}{\pi}}
  \approx k \left( 1 + i \frac{D}{2\pi} \right)$ for $D <<1$, 
the constant $D$ corresponds to the fraction of damping per wavelength.
E.g.\ if $D = 0.01$, then after 100 wavelengths, a wave amplitude is
reduced to $e^{-1}$ times the original amplitude.

In this section, $h_{\rm f}$ denotes the fine grid spacing. The coarse 
grid has spacing $2 h_{\rm f}$, and discrete Helmholtz operators
(both fine and coarse) are defined by 
\begin{equation}
  A u = D_{2,M} u - (k^2 + i \epsilon) u
\end{equation}
where $D_{2,M}$ is the standard central difference approximation of the second order
derivative of width $2M+1$ on the appropriate grid. By $A$ we denote
the fine and by $A_{\rm c}$ the coarse operator as previously.
The prolongation operator $I_{\rm P}$ will be given by linear interpolation and the
restriction $I_{\rm R}$ by full weighting
\cite{TrottenbergOosterleeSchueller2001}. For the smoother it is
assumed that  the
complex-shifted Helmholtz operator
$A_{\rm s} = A - i \alpha_{\rm s} k^2 I$ is inverted exactly, so no
domain decomposition since that would make the local Fourier analysis rather complicated.

\subsection{Local Fourier analysis\label{subsec:lfa_toy_methods}}

The smoother, the coarse-grid correction and the
two-grid cycle act linearly on the error $u - A^{-1} f$. The
corresponding operators will be denoted by $S$, $K$ and $M$ as in
\cite{TrottenbergOosterleeSchueller2001}. The point is to compute 
the asymptotic convergence rate $\rho(M)$.

Define the spaces
\begin{equation}
  \begin{aligned}
  E_{{\rm f},\theta}
  = {}& \operatorname{span} \left(
  e^{i \theta x / h_{\rm f}}, e^{i (\theta + \pi) x / h_{\rm f}}
\right) \subset V, \qquad \text{and}
  \\
    E_{{\rm c},\theta}
  = {}& \operatorname{span} \left(
  e^{i \theta x / h_{\rm f}} \right)  \subset V_{\rm c} .
\end{aligned}
\end{equation}
\NewInRevision{%
Here $\theta$ denotes a wave number scaled by $h$.}
The operators $A, A_{\rm s}, A_{\rm c}, I_{\rm P}, I_{\rm R}, S, K, M$ are of
course defined as maps
from one of $\{V, V_{\rm c} \}$ to one of $\{V, V_{\rm c} \}$.
By translation invariance, they also map one of $\{ E_{{\rm f},\theta}, E_{{\rm c},\theta} \}$
to one of $\{ E_{{\rm f},\theta}, E_{{\rm c},\theta} \}$. The matrices
for these maps will be called symbols and denoted $\widehat{A}(\theta)$,
$\widehat{A}_{\rm s}(\theta)$, etc. 

The derivation of these symbols is standard
\cite{TrottenbergOosterleeSchueller2001}.
The symbols $\widehat{A}(\theta)$,
$\widehat{A}_{\rm s}(\theta)$ are diagonal $2 \times 2$ matrices,
where the diagonal entries are the regular symbols (cf.\
(\ref{eq:symbol_sigma_P})), denoted here by $\widetilde{A}(\theta)$,
$\widetilde{A}_{\rm s}(\theta)$.
\NewInRevision{%
For example, for second order finite differences, the symbol is
\begin{equation}
  \widetilde{A}(\theta)
  = h^{-2} ( 2 - 2 \cos(\theta) ) - (k^2+i \epsilon)
  = h^{-2} 4 \sin(\theta/2)^2 - (k^2+i \epsilon) .
\end{equation}
}
The restriction and prolongation maps have symbols
\begin{equation}
  \begin{aligned}
    \widehat{I}_{\rm R}(\theta) = {}& 
    \begin{bmatrix} \frac{1}{2} + \frac{1}{2} \cos(\theta)
      \, & \,
      \frac{1}{2} - \frac{1}{2} \cos(\theta) 
    \end{bmatrix} ,
           \qquad \text{and}
           \\
  \widehat{I}_{\rm P}(\theta) = {}& \begin{bmatrix} \frac{1}{2} + \frac{1}{2} \cos(\theta) \\
    \frac{1}{2} - \frac{1}{2} \cos(\theta) 
  \end{bmatrix}
  \end{aligned}
\end{equation}
cf.\ \cite{TrottenbergOosterleeSchueller2001}.
The symbol of $S$ is given by
\begin{equation}
  \begin{aligned}
  \widehat{S} = {}& \operatorname{diag} [\widetilde{S}(\theta),
  \widetilde{S}(\theta+\pi) ] , \qquad \text{with}
  \\
  \widetilde{S}(\theta) = {}& 1 - \widetilde{A}_{\rm s}^{-1} \widetilde{A}(\theta)
\end{aligned}
\end{equation}
The matrix $\widehat{K}(\theta)$ satisfies
\begin{equation}
  \widehat{K}(\theta) = I - \widehat{I}_{\rm P}(\theta) \widehat{A}_{\rm c}^{-1}(\xi)  \widehat{I}_{\rm R}(\theta)
  \begin{bmatrix}
    \widetilde{A}(\theta) & 0 \\
    0 & \widetilde{A}(\theta + \pi) 
  \end{bmatrix} .
\end{equation}
while $\widehat{M}(\theta)$ is given by
\begin{equation}
  \widehat{M}(\theta) =
  \widehat{S}(\theta)^{\nu_2} \widehat{K}(\theta) \widehat{S}(\theta)^{\nu_1} .
\end{equation}
The asymptotic convergence rate of the two-grid method is
\begin{equation}
  \rho  = \sup_{\theta \in [-\frac{\pi}{2}, \frac{\pi}{2} )}
  \rho(\widehat{M}(\theta)) .
\end{equation}

\subsection{Dispersion errors}

Dispersion errors are defined with reference to the operator
$-\frac{d^2}{dx^2} - k^2 = 0$, i.e.\ with $\epsilon = 0$. In the kernel
of this operator are plane waves $e^{i\xi x}$ with $\xi = \pm k$.  In
general the discrete operators also have a kernel with plane wave
solutions on a grid. If $e^{i \xi_f x}$ is such a plane wave, the
(relative) dispersion error is defined by
\begin{equation}
  \text{(relative) dispersion error} = \frac{\xi_{\rm f}}{k} - 1 .
\end{equation}

In this section, it is convenient to work with scaled (normalized)
wave numbers $h \xi$, where $\xi$ is a regular wave number.
So $\zeta_{\rm f}$ and $\zeta_{\rm c}$ will denote
normalized wave numbers for propagating plane waves for the fine and
coarse operators, i.e. $\zeta_{\rm f} = h \xi_{\rm f}$. For the two-grid method,
the relative difference between $\zeta_{\rm f}$ and $\zeta_{\rm c}$ is
important for the convergence. This quantity will be denoted by 
\begin{equation} \label{eq:define_delta_disp_err}
  \delta = \frac{ \zeta_{\rm c} - \zeta_{\rm f} }{kh} ,
\end{equation}
which is more or less the relative difference between $\zeta_{\rm f}$
and $\zeta_{\rm c}$, since for a discretization to make sense the
approximate identity $\zeta_{\rm f} \approx kh$ must hold.

\subsection{Results for the toy example}

The main conclusion from this section is that the asymptotic
convergence rate approximately equals the ratio
\begin{equation}
  R := \frac{2 \pi | \delta |}{ D } .
\end{equation}
where $\delta$ and $D$ are the constants from
(\ref{eq:define_delta_disp_err}) and (\ref{eq:epsilon_ksquare_D}).

For second order and fourth order finite differences, the convergence
rates $\rho$ and the ratio $R$ are plotted in
Figure~\ref{fig:toy_example_1}. Here parameters were
$\alpha_{\rm s} = 0.2$ and $D = 0.01$, so 1 \% damping per wavelength. 
\begin{figure}
  \begin{center}
    \includegraphics[width=7cm]{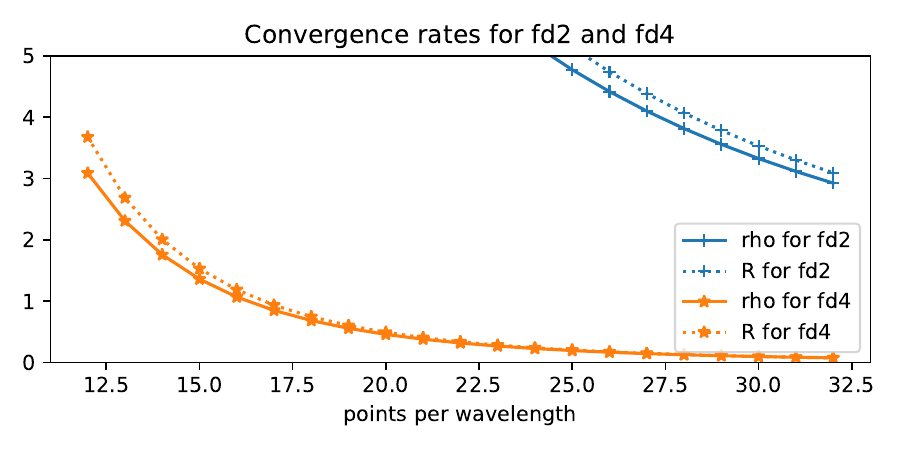}
  \end{center}
  \caption{$\rho$ and $R$ as a function of the number of points per
    wavelength. The example shows that $\rho$ is quite close to $R$.}\label{fig:toy_example_1}
\end{figure}
Note that $\rho$ and $R$ are indeed close to each other.

For a given $D$, significant wave energy propagates over distances
on the order of $1/D$ wavelengths. In such a case, the constant $\delta$ should
be made much smaller than $\frac{2\pi}{D}$ to have good convergence.
The hope is that this extends to bounded problems with absorbing
boundaries. For example in a constant coefficient case with absorbing
boundary conditions, waves can
propagate over at most distance $L := \operatorname{diam}(\Omega)$. Then
$\delta$ should be smaller than $\frac{4\pi^2}{k L}$. A full study of such
conjectures is outside the scope of this paper.

\subsection{An argument that $\rho \approx  R = \frac{2 \pi | \delta |}{ D }$}

We discuss why it can be expected that $\rho \approx R = \frac{2 \pi | \delta |}{ D }$.
This is done in two steps.
First, we will argue that in typical situations, for this 1-D example, the asymptotic
convergence rate $\rho$ is quite well approximated by
$\sup_{\theta \in [-\pi/2, \pi/2)}| \widehat{K}_{1,1}(\theta) |$. Then
we will argue that the quantity
$\sup_{\theta \in [-\pi/2, \pi/2)}| \widehat{K}_{1,1}(\theta) |$ in turn can be related
to the ratio
\begin{equation} \label{eq:define_R}
  R := \frac{2 \pi | \delta |}{ D } .
\end{equation}
Together this results in 
\begin{equation}
  \rho \approx \sup_{\theta \in [-\pi/2, \pi/2)}| \widehat{K}_{1,1}(\theta) |
  \approx R .
\end{equation}

Assuming $A$ and $A_{\rm c}$ are approximations of the true operator
$-\frac{d^2}{dx^2} - (k^2 + i \epsilon)$,
the symbols $\widehat{A}$ and $\widehat{A}_{\rm c}$
can be approximated around $\zeta_{\rm f}$ resp.\ $\zeta_{\rm c}$ by
\begin{equation}
  \widetilde{A} = 
\NewInRevision{ 
 \frac{d \Re \widetilde{A}}{dx}(\zeta_f) (\theta - \zeta_{\rm f})
  + O( (\theta - \zeta_{\rm f})^2 ) }
  \approx \frac{2k}{h} (\theta - \zeta_{\rm f}) - i
  \epsilon , 
\end{equation}
and
\begin{equation}
  \widehat{A}_{\rm c} =
\NewInRevision{ 
  \frac{d \Re \widehat{A}_{\rm c}}{dx}(\zeta_f) (\theta - \zeta_{\rm c})
  + O( (\theta - \zeta_{\rm c})^2 ) }
  \approx \frac{2k}{h} (\theta - \zeta_{\rm c}) -
  i \epsilon .
\end{equation}
\NewInRevision{Here it is used that the derivative
$\frac{d \Re \widetilde{A}}{dx}(\zeta_f)$ is approximately equal to
the derivative of the true symbol at $\theta = kh$  which is given by}
\[  
  {\color{red}
  \left. \frac{d}{d\theta} h^{-2} ( \theta^2 - k^2 h^2 ) \right|_{\theta = kh}
  = \left. 2 h^{-2} \theta \right|_{\theta = kh} = 2 h^{-1} k }
\]
\NewInRevision{and similarly for $\widehat{A}_{\rm c}(\theta)$.}
Therefore 
\begin{equation} \label{eq:K11hat_1}
  \widehat{K}_{1,1}(\theta) = 1 - \frac{1}{4} (1+ \cos \theta)^2 \widehat{A}_{\rm c}(\theta)^{-1}
  \widetilde{A}(\theta)
  \approx 1 - c \frac
  { \frac{2k}{h} ( \theta - \zeta_{\rm f} ) - i \epsilon }
  { \frac{2k}{h} ( \theta - \zeta_{\rm c} ) - i \epsilon }
  = 
  1 - c \frac{  \theta - \zeta_{\rm f}  - i \tilde{\epsilon} }
    {  \theta - \zeta_{\rm c}  - i \tilde{\epsilon} } ,
\end{equation}
where $c = \frac{1}{2} (1+ \cos (kh) )^2$ and
$\tilde{\epsilon} = \frac{\epsilon h}{2k}$ and in the last equality 
numerator and denominator were divided by $\frac{2k}{h}$.

The factor $c = \frac{1}{2} (1+ \cos (kh) )^2$ says something about the
prolongation and restriction. If the plane waves $e^{i x k}$ are
finely sampled (the number of points per wavelength is large) than $c$
is near one. Some values are $c \approx 0.56$ for 6 points per
wavelength (at the fine level), $c \approx 0.73$ for 8 points per
wavelength $c \approx 0.87$ for 12 points per wavelength.

It follows that $\widehat{K}_{1,1}(\theta)$ is near zero (good
convergence) if
$\zeta_{\rm f} = \zeta_{\rm c}$ and $c$ is near one. On the other
hand, if $|\tilde{\epsilon}| \ll | \zeta_{\rm c} - \zeta_{\rm f} |$,
then there are two problems. First it follows that
$| \widehat{K}_{1,1}(\theta) | \gg 1$ (divergence) for $\theta = \zeta_{\rm c}$,
because of division by a near-zero number. Secondly for $\theta =
\zeta_{\rm f}$ it follows that $\widehat{K}_{1,1}(\theta) \approx 1$
(very slow convergence at best). With equation (\ref{eq:K11hat_1}) one can
therefore intuitively understand the convergence properties. The maximum of
$| \widehat{K}_{1,1}(\theta) |$ can be determined:

\begin{lemma}
  Let
  \[
    f(\theta) =
    1 - c \frac{  \theta - \zeta_{\rm f}  - i \tilde{\epsilon} }
    {  \theta - \zeta_{\rm c}  - i \tilde{\epsilon} }    
  \]
  then, 
  \begin{equation} \label{eq:max_K11}
    \max_{\theta \in \RR}| f(\theta) |
    =
    \sqrt{ (1-c)^2 + \left( \frac{ | c R |}{2} \right)^2}
    +  \frac{ c R }{ 2 } .
  \end{equation}
\end{lemma}

\begin{proof}
  We let $\tilde{\theta} = \theta - \zeta_{\rm c}$. Then
  $f(\theta)$ can be rewritten as
  \begin{equation}
    f(\theta) 
  =  1 - c \frac{  \tilde{\theta} - \zeta_{\rm f} + \zeta_{\rm c}  - i \tilde{\epsilon} }
    {  \tilde{\theta}  - i \tilde{\epsilon} }
       = 1 - c + c \frac{ \zeta_{\rm f} - \zeta_{\rm c}
    }{\tilde{\theta} - i \tilde{\epsilon}} .
  \end{equation}
  When $\tilde{\theta}$ varies in $\RR$, then $f(\theta)$ maps out a circle in the
complex plane with center
$(1-c) + i \frac{ c(\zeta_{\rm f} - \zeta_{\rm c}) }{2\tilde{\epsilon}}$
and radius
$\frac{ c(|\zeta_{\rm f} - \zeta_{\rm c}|) }{2\tilde{\epsilon}}$.
The maximum absolute value of the points on the circle is
\begin{equation}
  \sqrt{ (1-c)^2 + \left( \frac{ c(\zeta_{\rm f} - \zeta_{\rm c})
      }{2\tilde{\epsilon}} \right)^2} +
  \frac{ c | \zeta_{\rm f} - \zeta_{\rm c} |  }{2\tilde{\epsilon}}
\end{equation}
The factor $\frac{|\zeta_{\rm f} - \zeta_{\rm c}| }{\tilde{\epsilon}}$ can
be rewritten as
\begin{equation}
  \frac{|\zeta_{\rm f} - \zeta_{\rm c}|}{\tilde{\epsilon}}
  = \frac{| \delta | h k}{\frac{\epsilon h}{2k}}
  = \frac{2 | \delta | k^2}{ \epsilon} = \frac{ 2 \pi | \delta |}{D}
  = R
\end{equation}
This leads to the result.
\end{proof}

We next discuss the relations of $\rho$ to $\sup_{\theta \in [-\pi/2, \pi/2)}
|\widehat{K}_{1,1}(\theta)|$ and to the ratio $R$ defined above.
The smoother symbol $\widetilde{S}(\theta)$ satisfies
$\left| \widetilde{S}(\theta) \right| \le 1$ in general. It satisfies
\begin{equation} \label{eq:S_near_1}
  \widetilde{S}(\theta) \approx 1 
\end{equation}
when $\theta \approx \pm k h$ and
\begin{equation} \label{eq:small_S_region}
  \left| \widetilde{S}(\theta) \right| \ll 1
\end{equation}
when $\left|\frac{|\theta|}{h}\right|^2 - k^2$ is large compared to
$\epsilon_{\rm s}$, in particular for $\theta \in [\pi/2, 3\pi/2)$.
As we saw, 
\begin{equation}
  \widehat{M}(\theta)
  =
  \begin{bmatrix} \widetilde{S}(\theta) & 0 \\
    0 & \widetilde{S}(\theta+\pi) 
  \end{bmatrix}^{\nu_2}
  \begin{bmatrix}
    \widehat{K}_{1,1}(\theta) & 
    \widehat{K}_{1,2}(\theta) \\
    \widehat{K}_{2,1}(\theta) & 
    \widehat{K}_{2,2}(\theta)
  \end{bmatrix}
  \begin{bmatrix} \widetilde{S}(\theta) & 0 \\
    0 & \widetilde{S}(\theta+\pi) 
  \end{bmatrix}^\nu_1
\end{equation}
By sufficiently many smoother
applications (in the examples $\nu_1 = \nu_2 =1$ often suffices),
$\widehat{M}_{1,1}(\theta)$ is the
dominant component, since (\ref{eq:small_S_region}) implies that other
components can be made small,
while (\ref{eq:S_near_1}) implies that
\begin{equation}
  \widehat{M}_{1,1}(\theta) \approx \widehat{K}_{1,1}(\theta)
\end{equation}
for $\theta \approx kh$.

We next relate this to $R$.
For $c = 1$, we get from the lemma
\[
  \max_{\theta \in [-\pi/2, \pi/2)} \widehat{K}_{1,1}(\theta) \approx R
\]
When $c$ varies from $1$ to $0$, then $\max_{\theta \in \RR}|
f(\theta) |$ varies between $R$ and 1. In general $\max_{\theta \in \RR}|
f(\theta) |$ is between roughly 1 and $R$. 
We already saw that, in our example $c$ is smaller but near 1, with
the difference being larger for more points per wavelength.
This explains that in this toy example, $\rho$ is quite close to $R$.

\section{Local Fourier analysis of the finite-element two-grid method\label{sec:lfa_full_method}}

Local Fourier analysis applied to finite-element problems is different from
that for finite-difference problems which was applied in the
previous section.  The main difference is that it uses Bloch waves
\cite{bloch1929quantenmechanik, kittel2018introduction} instead of the
Fourier modes. Bloch waves are natural and well known in PDE theory
and finite-element analysis
\cite{odeh1964partial,ainsworth2004discrete}, however, in the local
Fourier analysis community it appears the concept is not referred to,
even though related technology has been developed \cite{hemker2003two}.  In
the next section we will therefore describe in some detail how
local Fourier analysis is done using Bloch waves.

For the local Fourier analysis the Helmholtz problem on $\RR^d$
\begin{equation} \label{eq:Helmholtz_lfa_section}
  - \Delta u - \left( k^2 + i \epsilon \right) u = f(x)
\qquad\qquad \text{in $\RR^d$} ,
\end{equation}
with constant $k$ and $\epsilon$ is considered.
It is again assumed that
$\epsilon = \frac{D k^2}{\pi}$ where $D>0$ is the damping per
wavelength as explained below (\ref{eq:epsilon_ksquare_D}).
In the numerical LFA results, we chose $D = 0.01$ which roughly implies that 
waves propagate over on the order of 100 wavelengths in the problem to
be solved.
For the smoother it is again assumed that the complex-shifted Helmholtz
operator is inverted exactly.

The numerical LFA results at the end lead to several conclusions.
First, convergence is much better when optimized finite differences
are used at the coarse level, than when standard Galerkin
$p$-coarsening is used. Secondly, for each order the results show an
approximate lower bound for the number of degrees of freedom per
wavelength such that the method is still convergent. In
Figure~\ref{fig:rho_lfa_mgpar}(b) for example, one sees that for order
6, there is convergence for $\ge 7$ dofs per wavelength. For Galerkin
$p$-coarsening it can again be observed that $\rho \approx R$ (cf.\
section~\ref{subsec:lfa_toy_methods}, while for finite difference
coarse-level discretization the ratio $\frac{\rho}{R}$ is clearly
larger than 1. Also, in some cases parameters were identified that
provided some improvement compared to the standard choice mentioned in
section~\ref{sec:two-grid-method}.

\subsection{Local Fourier analysis using Bloch waves}

It was noted in subsection \ref{subsec:lfa_toy_methods}, cf.\
\cite{TrottenbergOosterleeSchueller2001} that the smoother, the
coarse-grid correction and the full two-grid cycle act as linear operators on
the residual. The action of these operators, denoted $S, K$ and
$M$, could be described by low-dimensional matrices. These facts
remain true, however, here this is done using a Bloch wave basis, and
not using a Fourier basis. This will be explained next.

Bloch wave computations are associated with translation symmetry over
a lattice of translations. In this case the lattice is $(h \ZZ)^d$. By
assumption, the mesh and the finite-element space are invariant under
such translations.

It is convenient if the finite-element basis is compatible with these translations.
A {\em unit cell} will mean a set of mesh entities, i.e.\
vertices, edges and faces in 2-D, such that translates of the unit
cell are disjoint and their union is the entire mesh. Here a
unit cell will consist of the mesh entities in a square of size $h$, such that the
left and lower edges, and the lower left corner are part of the unit
cell, while the right and top edges, and the other corner points, are
part of neighboring unit cells. Let $\phi_{0,i}$
denote the basis functions associated with the mesh entities from the
unit cell ($0$ is here an
element in $\ZZ^d$). It is assumed that the other basis functions are
given by the translates
\begin{equation}
  \phi_{\alpha,i} = T_{\alpha h} \phi_{0,i} , \qquad \alpha \in \ZZ^d
\end{equation}
where $T_s$ denotes translation over a vector $s \in \RR^d$, i.e.
\begin{equation}
  T_s u(x) = u (x-s) .
\end{equation}
Since the translates of the unit cell form a partition of the mesh, a
full basis is obtained in this way. For an element $L$ of $V'$, the
translates are defined by duality
\begin{equation}
  T_s L(v) = L(T_{-s} v) .
\end{equation}

Given a vector $\theta$,
a Bloch mode with (normalized) wave vector $\theta$ is a function $u$ such that
\begin{equation} \label{eq:bloch_mode_by_translation}
  u(x+s) = e^{i \theta \cdot s / h} u(x) ,
  \qquad
  \text{for all $s \in (h \ZZ)^d$.}
\end{equation}
This can also be written as 
\begin{equation} \label{eq:bloch_mode_by_translation2}
  T_s u = e^{-i \theta \cdot s / h} u ,
  \qquad
  \text{for all $s \in (h \ZZ)^d$} ,
\end{equation}
which is also a valid definition for elements of $V'$.
Alternatively it can be characterized as a product 
$u(x) = e^{i \theta \cdot x / h} v(x)$
where $v$ is periodic over the lattice.
It follows that the coefficients $u_{\alpha,j}$ of a Bloch vector
$u = \sum_{\alpha,j} u_{\alpha,j} \phi_{\alpha,j}$ satisfy
\begin{equation}
  u_{\alpha, j} = e^{i \theta \cdot \alpha} u_{0,j} .
\end{equation}

Spaces of Bloch waves are defined next:
\begin{definition}
Let $W$ be an element of $\{ V, V', V_{\rm c}, V_{\rm c}' \}$. 
The subspace of Bloch modes in $W$ with wave
vector $\theta$ will be denoted by $E(W, \theta)$.
\end{definition}

An operator will be called invariant under a set of translations
if it commutes with these translations.
A bilinear form will be called invariant under translations in the
lattice if 
\begin{equation}
  a(T_s u, T_s v) = a(u,v) , \qquad s \in \left( h \ZZ \right)^d .
\end{equation}
The operator $V \to V'$ derived from this bilinear form is then
automatically invariant under translations.  The bilinear form
(\ref{eq:Helmholtz_lfa_section}) is translation invariant since $k^2$
and $\epsilon$ are assumed constant (in fact $h$-periodicity would be
enough).

If $P$ is an operator $W_1 \to W_2$, with both
$W_j$ in $\{ V, V', V_{\rm c}, V_{\rm c}' \}$, then $P_{\alpha,\beta}$
will denote the matrix blocks. If $P$ is translation invariant, then
\begin{equation}
  P_{\alpha,\beta} = 
  P_{\alpha+\gamma,\beta+\gamma} .
\end{equation}
If $W_1$ and $W_2$ have the same number of degrees of freedom in the unit
cell, such a matrix is called a $d$-level block Toeplitz matrix.

A key result for local Fourier analysis is
\begin{theorem}
  Let $P : W_1 \to W_2$ be lattice translationally invariant, then
  $P$ maps $E(W_1,\theta)$ into $E(W_2,\theta)$.
\end{theorem}
\begin{proof}
  Suppose $T_s u = e^{-i \theta \cdot s/h} u$ for $s \in (h \ZZ)^d$
  \[
    T_s P u = P T_s u = e^{-i \theta \cdot s/h} P u ,
  \]
  so $Pu \in E(W_2, \theta)$.
\end{proof}

If $\phi_{\alpha, j}$ is the standard basis for a space then $b_j$,
$j = 1, \ldots, N_{\rm W}$ given by
\begin{equation} \label{eq:basis_E_W_theta}
  b_j = \sum_{\alpha \in \ZZ^d} e^{i\theta \cdot \alpha} \phi_{\alpha, j} .
\end{equation}
form a basis for $E(W, \theta)$. 

\begin{definition}
  Let $W_1$ and $W_2$ be spaces as above and $P : W_1 \to W_2$ be translation
  invariant. The symbol $\sigma(P)(\xi)$
  is defined to be the matrix of the map
  $P : E(W_1,\xi) \to E(W,\xi)$ with respect to the bases
  (\ref{eq:basis_E_W_theta})
  for $W_1$ and $W_2$.
  The symbol will also be denoted by $\widehat{P}(\theta)$.
\end{definition}

If $P : W \to W$ with $W$ an element of $\{ V, V', V_{\rm c}, V_{\rm
  c}' \}$, then
\begin{equation}
  \rho(P) = \sup_{\theta \in [-\pi,\pi)^d} \rho(\widehat{P}) .
\end{equation}
This will be used to approximate the asymptotic convergence rate of
the two-grid method.

The entries of $\sigma(P)$ satisfy
\begin{equation}
  \sigma(P)(\theta)_{j,k}
  = \left( P b_k(\theta) \right)_{0,j}
  = \sum_\beta e^{i \theta \cdot \beta} ( P_{0,\beta} )_{j,k} 
\end{equation}
so that a formula to compute $\sigma(P)$ is given by
\begin{equation} \label{eq:symbol_sum}
  \sigma(P)(\theta) = \sum_\beta e^{i \theta \cdot\beta} P_{0,\beta} .
\end{equation}

In this work operators are either compactly supported, in the sense that
$P_{\alpha,\beta} = 0$ if $\| \beta - \alpha \|$ is
larger than some constant or they are sums and 
products of such operators and their inverses. 
In case an operator is compactly supported , then
(\ref{eq:symbol_sum}) can be used directly  to
compute its symbol. Symbols of inverses of such operators can be computed by
using that
\begin{equation} \label{eq:symbol_of_inverse}
  \sigma(P^{-1}) = \sigma(P)^{-1} ,
\end{equation}
and symbols of sums or products of operators by taking the
corresponding sum or product of the symbols.

As noted in
section~\ref{subsec:lfa_toy_methods}, there are operators $S, K$ and
$M$ which correspond to the actions of the smoother, coarse-grid correction and the
two-grid iteration on the residual $f - A u$. In local Fourier
analysis one is interested in the asymptotic convergence factor given by
\begin{equation}
  \sup_{\theta \in [-\pi,\pi)^d} \rho(\widehat{M}(\theta)) .
\end{equation}
The symbol $\widehat{M}(\theta)$ is computed using the following formulas
\begin{align}
  \widehat{S}(\theta) = {}& I - \widehat{A}_{\rm s}(\theta)^{-1} \widehat{A}(\theta) 
  \\
  \widehat{K}(\theta) = {}& I - \widehat{I}_{\rm P}(\theta)
                            \widehat{A}_{\rm c}(\theta)^{-1}
                            \widehat{I}_{\rm R}(\theta)
                            \widehat{A}(\theta)
  \\
  \widehat{M}(\theta) = {}& \widehat{S}(\theta)^{n_{\rm s}}
  \widehat{K}(\theta) \widehat{S}(\theta)^{n_{\rm s}}
\end{align}
The operators $A, A_{\rm s}, A_{\rm c}, I_{\rm P}$ and $I_{\rm R}$ are
compactly supported so their symbols can be computed directly using
(\ref{eq:symbol_sum}).

\subsection{Dispersion relations for regular mesh finite elements\label{subsec:dispersion_fem}}

For the Helmholtz equation, the dispersion relation is the set of wave
numbers $\xi$ such that $e^{ix \cdot \xi}$ is a solution to
$-\Delta u - k^2 = 0$. It is the set  $\|\xi\| = k$. For finite elements,
the discrete dispersion relation is about the Bloch waves $u$ on an infinite
mesh such that $A u = 0$, where $A$ is the finite-element
matrix. These can be called propagating waves.
It can be determined by solving $\det \widehat{A}(\theta) = 0$, hence
by finding those $\theta$ such that $\widehat{A}(\theta)$ has at least
one zero eigenvalue.

To define a dispersion error, a wavenumber must be associated with
such a Bloch wave vector $\theta$. The dispersion error is then the
difference between this wave vector and the nearest `true' wave
vector. Because $\theta$ is really an equivalence class of wave
vectors (module a vector in $(2\pi \ZZ)^d$, see the appendix) there is
some ambiguity in the choice of such a wave vector. We choose the
member of the equivalence class that minimizes the dispersion error.
For a given mesh and discretization waves propagate in many
directions. The dispersion errors of this mesh (with a given number of
dofs per wavelength) is defined as the maximum of the dispersion error
over all discrete propagating waves%
\footnote{Alternatively one might define it  as the maximum over all
  true wave vectors of the distance to the nearest discrete
  propagating wave vector. A discussion of such issues is outside the
  scope of this paper.}.

The computation of discrete dispersion relations, not in our 2-D
setting, was discussed in \cite{ainsworth2004discrete}. In the
appendix a numerical computation is explained that does apply to the present
setting.  Figure~\ref{fig:dispersionErrors} shows the dispersion
errors for quad and tri meshes for the various finite-element schemes
(labeled `gal, $p$', with $p$ the order)
and the optimized finite differences (labeled `opt').
\begin{figure}
  \begin{center}
    (a) \hspace*{66mm}    (b)\\
    \includegraphics[width=70mm]{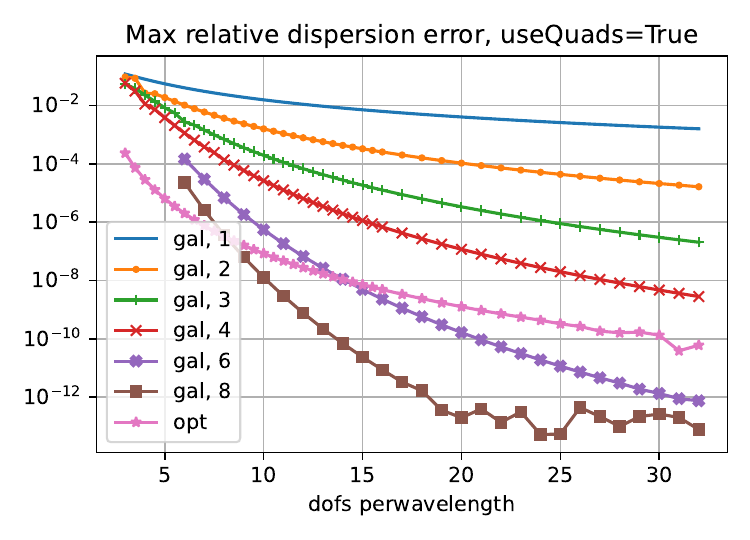}
    \includegraphics[width=70mm]{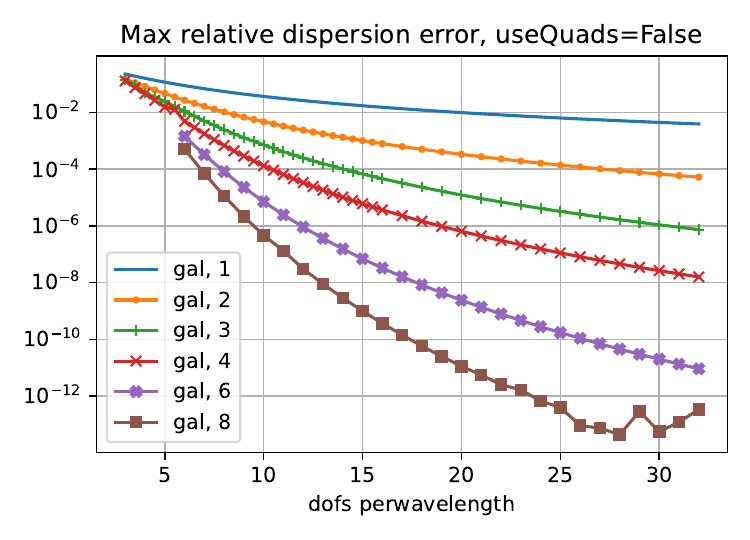}
    \end{center}
  \caption{Dispersion errors as a function of dofs per wavelength for
    (a) regular mesh of squares; and (b) regular mesh of triangles}\label{fig:dispersionErrors}
\end{figure}

\subsection{Asymptotic convergence rates from local Fourier analysis}

\begin{figure}
  (a) \hspace*{66mm} (b) \\
\includegraphics[width=70mm]{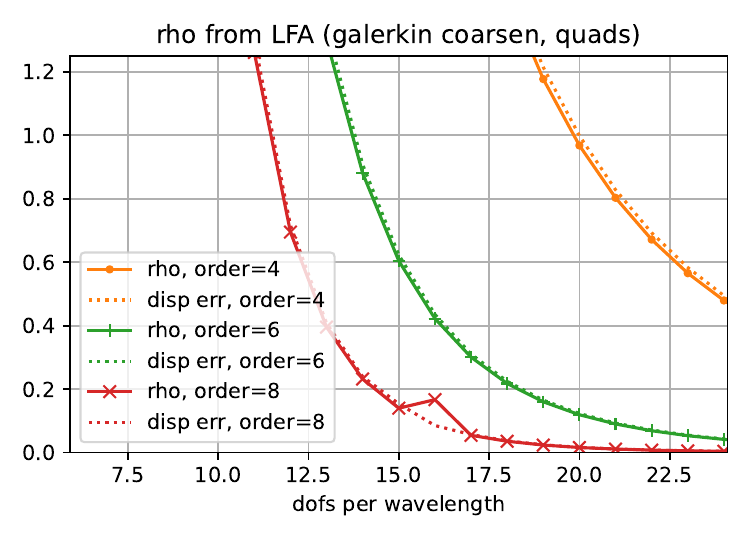}
\includegraphics[width=70mm]{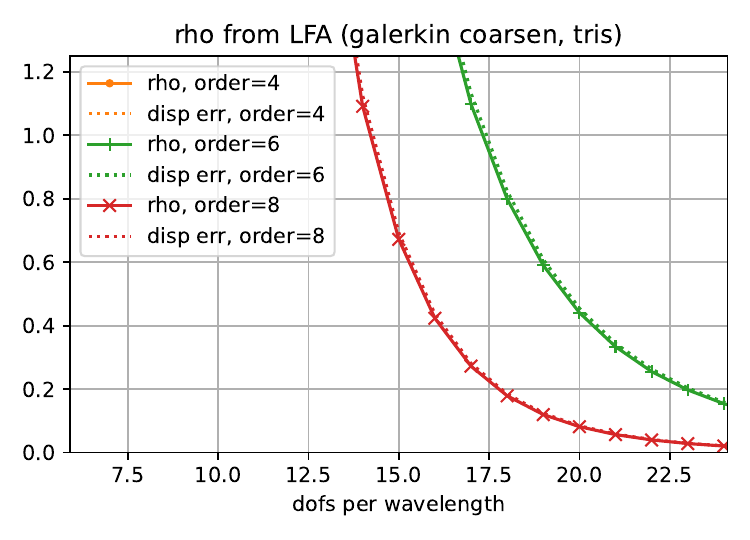}\\
  (c) \hspace*{66mm} (d) \\
\includegraphics[width=70mm]{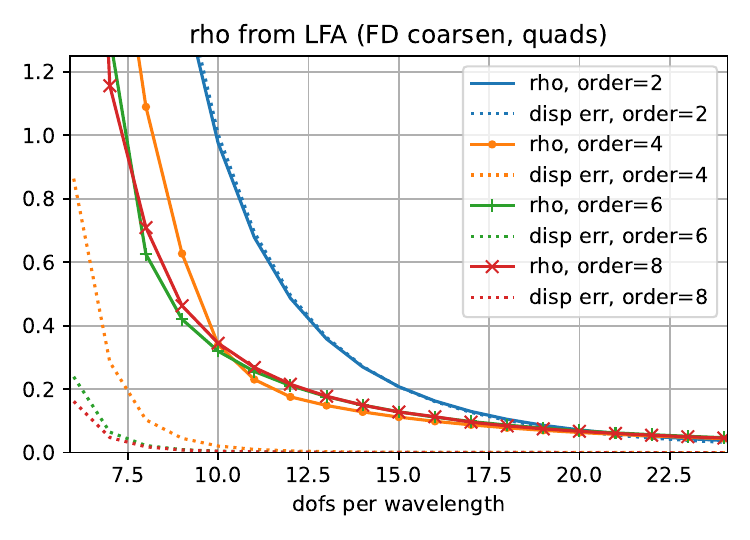}
\includegraphics[width=70mm]{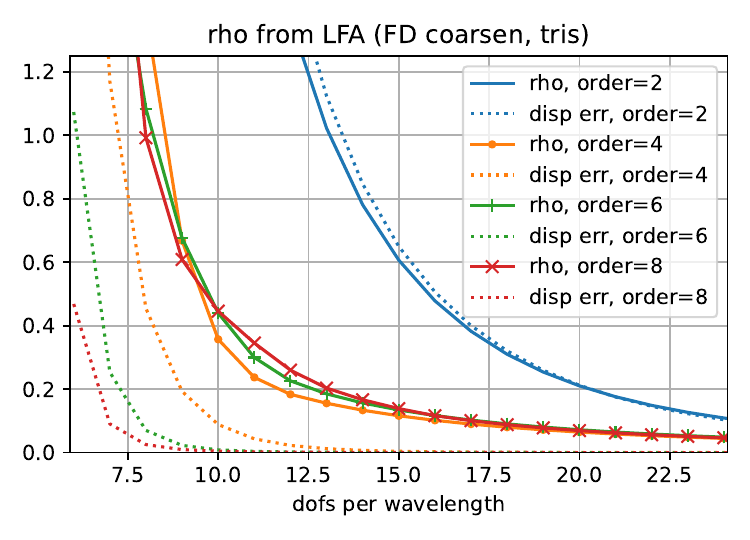}
  \caption{Convergence rate for standard Galerkin $p$-type coarsening (a,b)
    and optimized FD coarsening (c,d). Results are shown for square meshes
    (a,c) and triangular meshes (b,d). Square meshes and optimized FD
    coarsening results in better convergence, so that the method can
    be used with a smaller number of dofs per wavelength. In (a)
    order 2 curve is above the plot area, in (b) orders 2 and 4 are
    above the plot area.}
  \label{fig:rho_lfa_fd_vs_pcoarsen}
\end{figure}

\begin{figure}
  (a) \hspace*{66mm} (b) \\
\includegraphics[width=70mm]{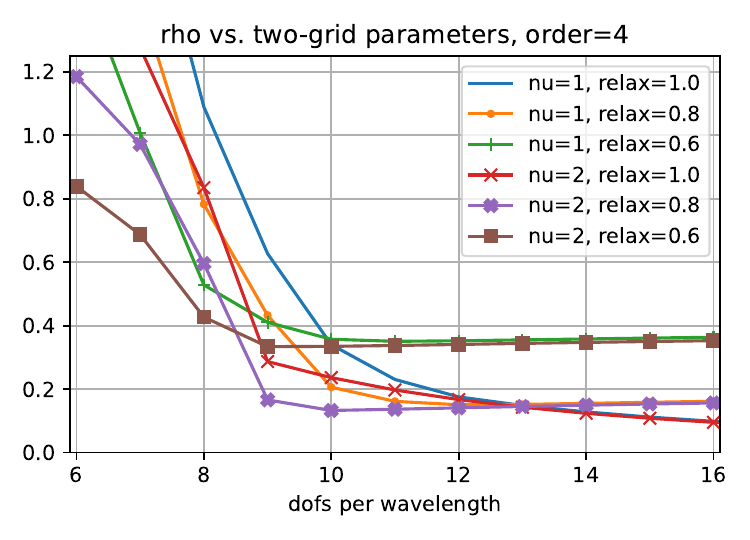}
\includegraphics[width=70mm]{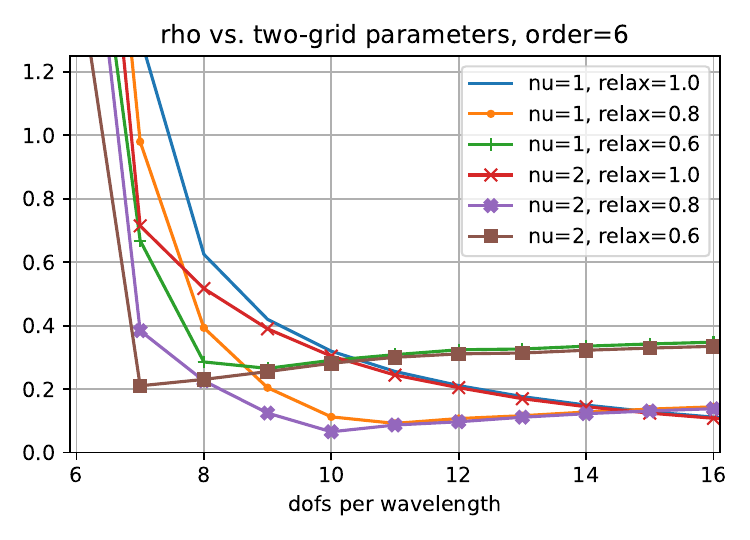}\\
  (c)\\
\includegraphics[width=70mm]{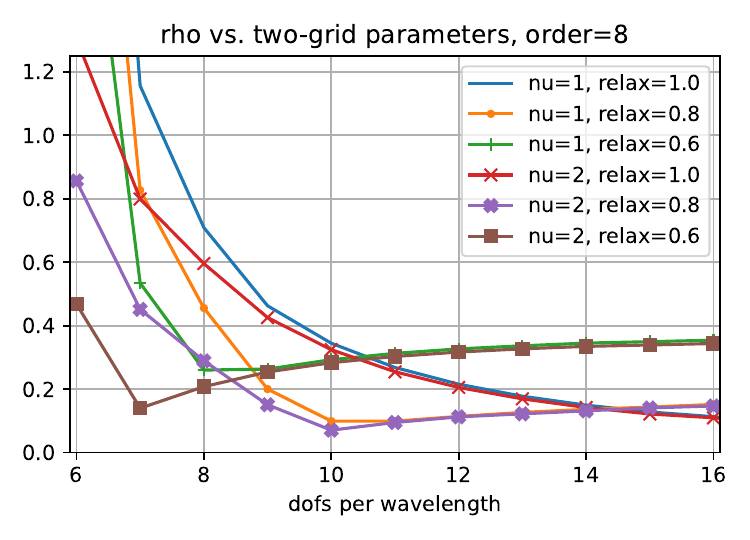}
  \caption{Convergence rate for optimized FD coarsening for different
    order and different multigrid parameters $\nu$ and relaxation
    factor $\omega_{\rm CGC}$. For low numbers of dofs per wavelength
    convergence is improved by increasing $\nu$ and choosing a
    relaxation factor below 1.}
  \label{fig:rho_lfa_mgpar}
\end{figure}

To obtain asymptotic convergence rates, the symbol computation methods
described in subsection \ref{sec:lfa_full_method} were implemented
using the Python programming language and the finite-element package
NGSolve. Basis transformations were implemented on top of NGSolve
finite-element spaces and used to obtain the matrix blocks $P_{0,\beta}$ in equation
(\ref{eq:symbol_sum}) where $P$ stands for the various operators.
The advantage of building this on top of NGSolve was that results
could easily be obtained for different orders of finite elements and
tri and square and meshes, because the package contains all these 
element types.

A first set of results is displayed in
Figure~\ref{fig:rho_lfa_fd_vs_pcoarsen}. Convergence rates are shown
as a function of the number of dofs per wavelength, which equals $\frac{2
\pi p}{k h}$. The figure shows that
convergence rates for Galerkin $p$-type coarsening are significantly
worse than those for optimized FD coarsening.
For Galerkin $p$-type coarsening the convergence rate is more or less
given by the estimated maximum dispersion error, which is in line with
results of the previous section. For the optimized FD coarsening,
convergence rates are larger than the dispersion errors except for
$p=2$. This may have to do with the fact that for optimized FD
coarsening, $p \ge 4$, the coarse space is not a subspace of the fine
space.
For FD coarsening, the convergence at low numbers of dofs per
wavelength can be somewhat improved by
increasing the number of smoother steps from 1 to 2 or by introducing a
relaxation factor in the coarse-grid correction, see
Figure~\ref{fig:rho_lfa_mgpar}.

The parameter $\alpha_{\rm s}$ was set to 0.2. The local Fourier
analysis shows that this value is small enough. Choices for this
parameter are best studied in conjunction with the domain decomposition
parameters, which is done in the next section.

\section{Numerical convergence results\label{sec:numerical_results}}

In this section numerical convergence results are presented. This has
multiple goals. One goal is to establish reasonable choices for the
parameters of the method listed in
subsection~\ref{subsec:parameters_of_method}, in addition to what was
found in the previous section. The other goals are to test the
performance of the new method and to compare it with a method using
standard Galerkin coarsening.  The performance is tested for many
different cases. The examples feature different boundary conditions,
different order finite elements, different problem sizes, and
different numbers of degrees of freedom per wavelength. All tests are
on the unit square with constant $k$, $\epsilon = 0$, and absorbing
boundary conditions or absorbing layers on at least two non-opposite
sides to avoid standing wave solutions.  Results for the new method
will be labeled `opt'. Results for the domain decomposition smoother
with Galerkin coarsening are labeled `gal'. To reduce the multitude
of results somewhat, only results for quad meshes are shown.

Implementation was done in Python, using NGSolve to obtain finite
element matrices, and scipy.sparse for sparse matrix computations.  In
case of an absorbing layer, a number of cells corresponding to about
40 dofs were added and over this distance, the complex contribution
$\epsilon$ increased with a sin-square profile from $0$ to
$\frac{2 k^2}{\pi}$, meaning $D=2$ damping per wavelength in the
notation of subsection~\ref{subsec:damped_Helmholtz_for_LFA}.

In Table~\ref{tab:its_vs_pars_OPT}, different choices for the
parameters of the numerical scheme were tested. In part (a), the
following parameters were varied: the complex shift for the smoother
$\alpha_{\rm s}$, the subdomain size for the domain decomposition,
and the number $n_{\rm dd}$ of domain decomposition iterations per
smoother step%
\footnote{Note that a priori it makes more sense to use one
domain decomposition iteration per smoother step and vary the number
of smoother steps, but we varied this parameter anyway. This is
because the smoother aims to reduce the residual of the actual problem
and not the complex-shifted problem.}.
For the subdomain sizes, the values around 4 to 6 approximately minimize
the total computation cost per smoother step in view of the cost
scaling and overlaps.
The conclusion is that there is not much dependence on these
parameters and that $\alpha_{\rm s} = 0.2$, $n_{\rm dd} = 1$ and
both choices for the subdomain size are reasonable. For order $p=2$
and $p=8$ essentially the same behavior was observed (this is not shown).
In part (b) of this table, smoother parameters were varied for
different numbers of degrees of freedom per wavelength (denoted ppw)
and orders $p=4$ and $p=6$. The conclusion is that for 8 or more
degrees of freedom per wavelength, the standard choices $n_{\rm s} =
1$ and $\omega_{\rm c} = 1$ are fine, while for $p=6$ and 6 degrees of
freedom per wavelength the number of coarse-level solves can be
reduced by using $n_{\rm s} = 2$ and $\omega_{\rm c} = 0.8$. For order
$p=8$ similar behavior as for $p=6$ was observed (for orders 2 and 4,
it does not make sense to use 6 degrees of freedom per wavelength due
to dispersion errors).

Note that for low orders the subdomains are smaller, in the sense of
physical length and number of degrees of freedom. Apparently this does
not affect the convergence much, presumably because long range wave
propagation is handled by the coarse-level solver.

For Galerkin coarsening we observed that the subdomain size has some
influence. In order to have a similar behavior for different values of
the order $p$, the subdomain size was fixed at around 40 dofs in each
direction, and the number of subdomain mesh cells in each direction
was fixed accordingly from 5 cells for order 8 up to 20 mesh cells for
order 2.

For Galerkin coarsening the results in table~\ref{tab:its_vs_pars_GAL}
show that in general it is
beneficial to introduce a complex shift in the coarse-level
operator. Optimal values, however, depend on the number of degrees of
freedom per wavelength and we did not find a choice optimal for all
values of the number of degrees of freedom per wavelength that were
tested. In the convergence tests in
Figure~\ref{fig:its_varying_examples1} we therefore
consider two possibilities: $\alpha_{\rm c} = 0.02$ and no coarse-grid
correction. We chose $\alpha_{\rm s} = 0.02$.

\begin{table}
  \begin{center}
    (a)\\
\begin{tabular}{|c|ccccc||ccccc|} \hline
  & $\alpha_{\rm s}$ = 0.1 & 0.15 & 0.2 & 0.3 & 0.4 &
   0.1 & 0.15 & 0.2 & 0.3 & 0.4\\ \hline 
  & \multicolumn{5}{c||}{$p=4$, subdomsize=4} & 
  \multicolumn{5}{c|}{$p=4$, subdomsize=6}\\ \hline
  $n_{\rm dd}$=1 & 7 & 7 & 7 & 7 & 7 & 7 & 6 & 6 & 6 & 7 \\                                   
  2 & 6 & 6 & 6 & 6 & 6 &  6 & 5 & 5 & 6 & 6 \\ \hline \hline
  & \multicolumn{5}{c||}{$p=6$, subdomsize=4}
  & \multicolumn{5}{c|}{$p=6$, subdomsize=6}\\ \hline
$n_{\rm dd}$=1 & 
  9 & 8 & 8 & 8 & 8  & 9 & 8 & 8 & 8 & 8 \\
2 &   7 & 7 & 7 & 7 & 8 &  7 & 7 & 7 & 7 & 8\\ \hline
\end{tabular} \\

\medskip

(b)\\
\begin{tabular}{|c|c||ccc||cccc|} \hline
  && \multicolumn{3}{c||}{$p=4$} & 
  \multicolumn{4}{c|}{$p=6$}\\ \hline
  && ppw=8 & 10 & 14
  & 6 & 8 & 10 & 14 \\ \hline
  & $\omega_{\rm c}$ = 1.0     & 9 & 7 & 5     & 19 & 8 & 7 & 5\\       
 $n_{\rm s}$ = 1 & 0.8      & 8 & 6 & 7     & 16 & 6 & 5 & 7\\    
  & 0.6                   & 11 & 11 & 12  & 15 & 8 & 10 & 11\\   \hline
  & $\omega_{\rm c}$ = 1.0     & 8 & 6 & 5     & 16 & 8 & 7 & 5\\    
 $n_{\rm s}$ = 2 & 0.8      & 7 & 6 & 7     & 12 & 5 & 5 & 7\\    
  & 0.6                   & 11 & 11 & 12  & 11 & 8 & 10 & 11    \\ \hline
\end{tabular}
\end{center}
\caption{Iteration counts vs.\ parameter choices for `OPT'
  coarsening,
  (a) varying the complex shift parameter
    $\alpha_{\rm s}$, domain decomposition iterations $n_{\rm dd}$,
    and subdomain size.
    For order $p=4$, results are for 10 ppw, for order $6$ for 8 ppw.
    We choose $n_{\rm dd} = 1$ and $\alpha_{\rm
      s} = 0.2$ for further computations. Dependence on these
    quantities is weak. (b) varying number of smoother steps
    $n_{\rm s}$ and coarse-grid correction relaxation factor
    $\omega_{\rm c}$.\label{tab:its_vs_pars_OPT}}
\end{table}
\begin{table}
  \begin{tabular}{|c|cccc||cccc|} \hline 
 & \multicolumn{4}{c||}{$p=6$, ppw=6} 
 & \multicolumn{4}{c|}{$p=6$, ppw=8} \\ \hline
 & $\alpha_{\rm s}$ = 0.01 & 0.02 & 0.05 & 0.1&  0.01 & 0.02 & 0.05 & 0.1\\ \hline 
  $\alpha_{\rm c}$=
0.0   & $>200$ & $>200$ & 110 & $>200$ & 15 & 15 & 14 & 15\\
0.01  & $>200$ & $>200$ & 61  & 102 & 14 & 14 & 14 & 14\\
0.02  & $>200$ & 177 & 46  & 73  & 14 & 13 & 13 & 14\\
0.05  & $>200$ & 30  & 31  & 47  & 15 & 15 & 15 & 16\\
0.1   & 35  & 21  & 25  & 35  & 18 & 18 & 18 & 19\\
0.2   & 21  & 19  & 21  & 29  & 22 & 21 & 21 & 23\\
0.5   & 21  & 20  & 22  & 25  & 25 & 24 & 25 & 28\\
noCGC & 23  & 22  & 24  & 29  & 29 & 27 & 28 & 32\\ \hline
\end{tabular} 
\caption{Iteration counts vs.\ parameter choices for `GAL'
  coarsening, varying complex shift parameters for the smoother
  and coarse-grid correction.\label{tab:its_vs_pars_GAL}}
\end{table}
\begin{figure}
  \begin{center}
    \hspace*{3mm} (a) \hspace*{33mm} (b)  \hspace*{33mm} (c) \hspace*{0mm}\\
    \includegraphics[height=32mm]{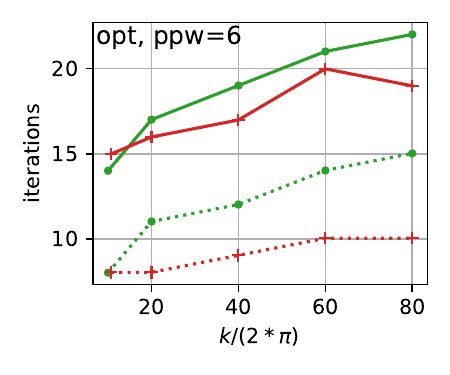}
    \includegraphics[height=32mm]{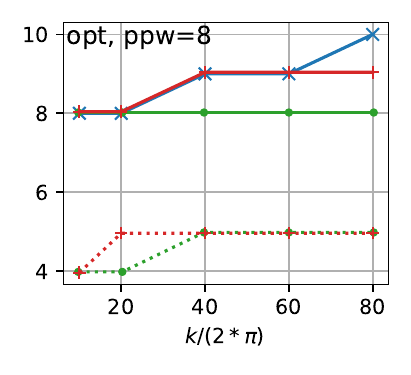}
    \includegraphics[height=32mm]{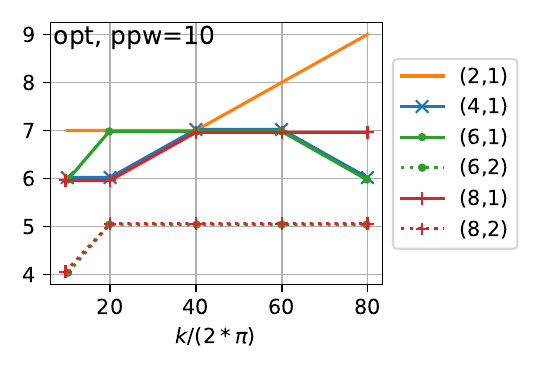}\\
    \hspace*{3mm} (d) \hspace*{33mm} (e)  \hspace*{33mm} (f) \hspace*{0mm}\\
    \includegraphics[height=32mm]{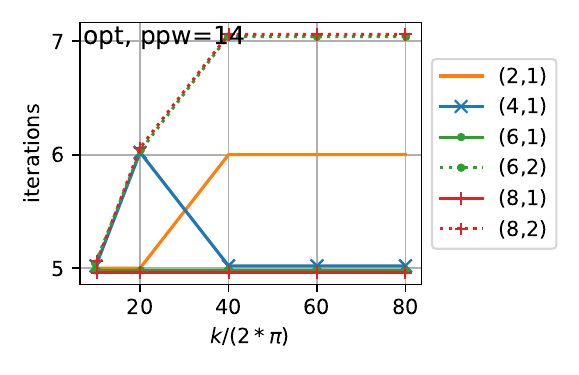}
    \includegraphics[height=32mm]{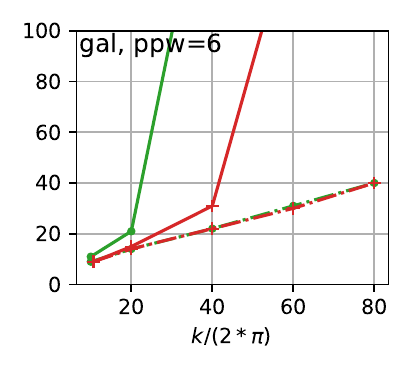}
    \includegraphics[height=32mm]{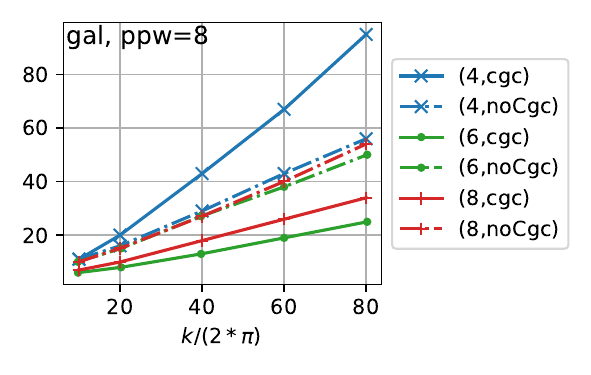}\\
    \hspace*{3mm} (g) \hspace*{33mm} (h) \\
    \includegraphics[height=32mm]{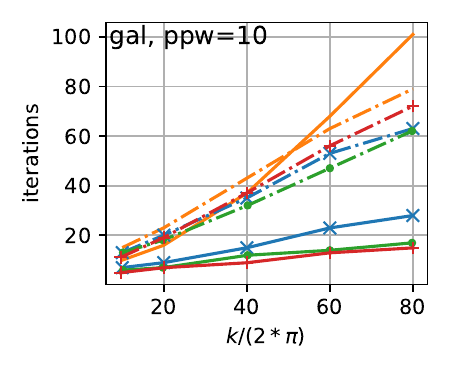}
    \includegraphics[height=32mm]{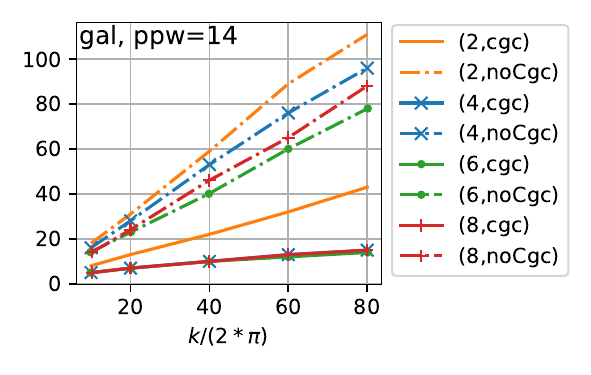}
  \end{center}
  \caption{Convergence vs.\ problem size. (a-d) `OPT' coarsening,
    6,8,10,14 dofs per wavelength with two choices of two-grid parameters (see main text);
    (e-h) for `GAL' coarsening with and without coarse-grid
    correction (see main text). `OPT' converge is 
    substantially better.\label{fig:its_varying_examples1}}
\end{figure}
\begin{figure}
  \begin{center}
    \hspace*{-11mm} (a) \hspace*{34mm} (b) \hspace*{2mm}\\
    \includegraphics[height=36mm]{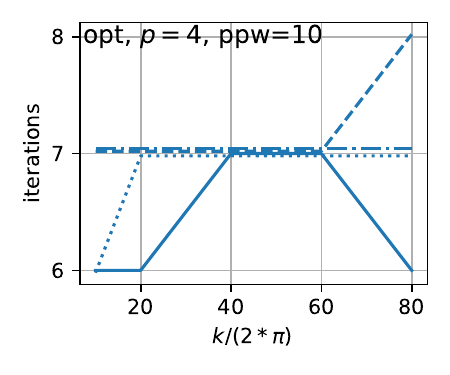}
    \includegraphics[height=36mm]{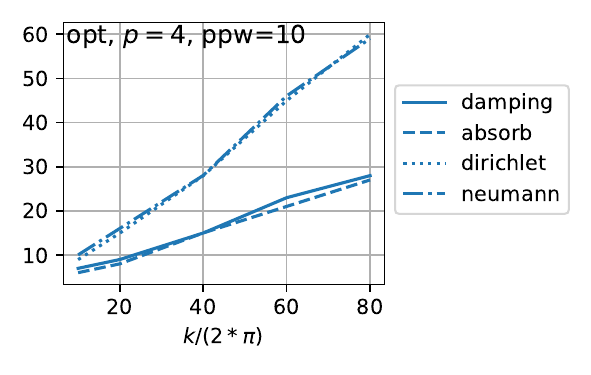}
  \end{center}
  \caption{Convergence for different choices of boundary condition (a)
  'OPT' coarsening: convergence more or less independent of boundary
  condition; (b) `GAL' coarsening: iteration count roughly doubles if
  Dirichlet or Neumann boundary conditions are present on two
  non-opposite sides.\label{fig:conv_bdy_cond}}
\end{figure}

In Figure~\ref{fig:its_varying_examples1} the convergence the method
is shown as a function of problem size for different values of the
finite-element order and the 
number of DOFs per wavelength. On the $x$-axis is the quantity
$\frac{k}{2\pi}$ which equals the domain size in terms of the number
of wavelengths. For the new method, the iteration counts are low and
have weak dependence on problem size. When Galerkin coarsening is
used, iteration counts are higher and generally increase linearly with
problem size.
In Figure~\ref{fig:conv_bdy_cond}, different boundary conditions were
tested. Introducing Neumann or Dirichlet
boundary conditions on two non-opposite sides has more or less the
same effect as doubling the problem size.

\section{Discussion and conclusions\label{sec:discussion_conclusions}}

In this work we have tested a new two-grid method for finite element
Helmholtz problems, and compared it to a variant using standard
Galerkin coarsening but with the same novel smoother. 

The method differs from previous method in several ways. The main
difference with previous two-level domain decomposition methods (see
\cite{graham2017domain} and references) is in the choice
of coarse mesh and the choice of coarse level operator. The number of
degrees of freedom is reduced by a factor 2 in each direction and
dispersion error minimizing finite differences are chosen as coarse
level method. Compared to multigrid methods, the choice of smoother
and the choice of coarse level discretization are new.

The main conclusion is that the new method converges well according to local
Fourier analysis, while numerical iterations counts are small and more
or less independent of problem size.

The comparison with dispersion errors, i.e.\ the discussion of
section~\ref{sec:lfa_toy}, and the data on $R$ in
Figures~\ref{fig:rho_lfa_fd_vs_pcoarsen}
and~\ref{fig:rho_lfa_mgpar}
show
that the dispersion errors strongly correlate with the good or poor
convergence. In case of Galerkin coarsening convergence rate curves
follow closely the dispersion errors curves.  In the new
method, the dispersion errors (the ratio $R$) is lower than the
asymptotic convergence rates. Our conjecture is that this is caused by
the fact that the coarse finite-element space is not a subspace of the
fine finite-element space.

The implication for the design of multigrid methods is that it is
essential for good convergence that the dispersion errors of the
coarse level method are small (more precisely the dispersion
relations of coarse and fine level schemes closely match).

\smallskip

\section*{Declarations}

\subsection*{Funding}

No funding is to be declared.

\subsection*{Conflict of interest statement}

The authors declared that they have no conflict of interest.

\subsection*{Acknowledgement}

Not applicable.

\bibliographystyle{abbrv}
\bibliography{helmSmooth}

\begin{thebibliography}{10}

\bibitem{ainsworth2004discrete}
M.~Ainsworth.
\newblock Discrete dispersion relation for hp-version finite element
  approximation at high wave number.
\newblock {\em SIAM Journal on Numerical Analysis}, 42(2):553--575, 2004.

\bibitem{appelo2020waveholtz}
D.~Appelo, F.~Garcia, and O.~Runborg.
\newblock Waveholtz: Iterative solution of the helmholtz equation via the wave
  equation.
\newblock {\em SIAM Journal on Scientific Computing}, 42(4):A1950--A1983, 2020.

\bibitem{BabuskaEtAl1995}
I.~Babu{\v{s}}ka, F.~Ihlenburg, E.~T. Paik, and S.~A. Sauter.
\newblock A generalized finite element method for solving the {H}elmholtz
  equation in two dimensions with minimal pollution.
\newblock {\em Comput. Methods Appl. Mech. Engrg.}, 128(3-4):325--359, 1995.

\bibitem{babuska1997pollution}
I.~M. Babuska and S.~A. Sauter.
\newblock Is the pollution effect of the {FEM} avoidable for the {Helmholtz}
  equation considering high wave numbers?
\newblock {\em SIAM Journal on numerical analysis}, 34(6):2392--2423, 1997.

\bibitem{bloch1929quantenmechanik}
F.~Bloch.
\newblock {\"U}ber die quantenmechanik der elektronen in kristallgittern.
\newblock {\em Zeitschrift f{\"u}r physik}, 52(7-8):555--600, 1929.

\bibitem{brandt1977multi}
A.~Brandt.
\newblock Multi-level adaptive solutions to boundary-value problems.
\newblock {\em Mathematics of computation}, 31(138):333--390, 1977.

\bibitem{brenner2008mathematical}
S.~C. Brenner and R.~L. Scott.
\newblock {\em The mathematical theory of finite element methods}.
\newblock Springer, 2008.

\bibitem{CalandraEtAl2013}
H.~Calandra, S.~Gratton, X.~Pinel, and X.~Vasseur.
\newblock An improved two-grid preconditioner for the solution of
  three-dimensional {H}elmholtz problems in heterogeneous media.
\newblock {\em Numer. Linear Algebra Appl.}, 20(4):663--688, 2013.

\bibitem{cocquet2024asymptotic}
P.-H. Cocquet and M.~J. Gander.
\newblock Asymptotic dispersion correction in general finite difference schemes
  for {Helmholtz} problems.
\newblock {\em SIAM Journal on Scientific Computing}, 46(2):A670--A696, 2024.

\bibitem{cocquet2021closed}
P.-H. Cocquet, M.~J. Gander, and X.~Xiang.
\newblock Closed form dispersion corrections including a real shifted
  wavenumber for finite difference discretizations of 2d constant coefficient
  helmholtz problems.
\newblock {\em SIAM Journal on Scientific Computing}, 43(1):A278--A308, 2021.

\bibitem{engquist2011sweeping}
B.~Engquist and L.~Ying.
\newblock Sweeping preconditioner for the {Helmholtz} equation: moving
  perfectly matched layers.
\newblock {\em Multiscale Modeling \& Simulation}, 9(2):686--710, 2011.

\bibitem{erlangga2006novel}
Y.~A. Erlangga, C.~W. Oosterlee, and C.~Vuik.
\newblock A novel multigrid based preconditioner for heterogeneous helmholtz
  problems.
\newblock {\em SIAM Journal on Scientific Computing}, 27(4):1471--1492, 2006.

\bibitem{gander2015applying}
M.~J. Gander, I.~G. Graham, and E.~A. Spence.
\newblock Applying gmres to the helmholtz equation with shifted laplacian
  preconditioning: what is the largest shift for which wavenumber-independent
  convergence is guaranteed?
\newblock {\em Numerische Mathematik}, 131(3):567--614, 2015.

\bibitem{gander2019class}
M.~J. Gander and H.~Zhang.
\newblock A class of iterative solvers for the helmholtz equation:
  Factorizations, sweeping preconditioners, source transfer, single layer
  potentials, polarized traces, and optimized schwarz methods.
\newblock {\em Siam Review}, 61(1):3--76, 2019.

\bibitem{graham2017domain}
I.~Graham, E.~Spence, and E.~Vainikko.
\newblock Domain decomposition preconditioning for high-frequency {Helmholtz}
  problems with absorption.
\newblock {\em Mathematics of Computation}, 86(307):2089--2127, 2017.

\bibitem{grote2019controllability}
M.~J. Grote and J.~H. Tang.
\newblock On controllability methods for the helmholtz equation.
\newblock {\em Journal of Computational and Applied Mathematics}, 358:306--326,
  2019.

\bibitem{hemker2003two}
P.~W. Hemker, W.~Hoffmann, and M.~Van~Raalte.
\newblock Two-level fourier analysis of a multigrid approach for discontinuous
  galerkin discretization.
\newblock {\em SIAM Journal on Scientific Computing}, 25(3):1018--1041, 2003.

\bibitem{ihlenburg1995dispersion}
F.~Ihlenburg and I.~Babu{\v{s}}ka.
\newblock Dispersion analysis and error estimation of galerkin finite element
  methods for the helmholtz equation.
\newblock {\em International journal for numerical methods in engineering},
  38(22):3745--3774, 1995.

\bibitem{kimn2013shifted}
J.-H. Kimn and M.~Sarkis.
\newblock Shifted {Laplacian} {RAS} solvers for the {Helmholtz} equation.
\newblock In {\em Domain decomposition methods in science and engineering XX},
  pages 151--158. Springer, 2013.

\bibitem{kittel2018introduction}
C.~Kittel and P.~McEuen.
\newblock {\em Introduction to solid state physics}.
\newblock John Wiley \& Sons, 2018.

\bibitem{knibbe20133d}
H.~Knibbe, C.~Oosterlee, and C.~Vuik.
\newblock 3d helmholtz krylov solver preconditioned by a shifted laplace
  multigrid method on multi-gpus.
\newblock In {\em Numerical Mathematics and Advanced Applications 2011}, pages
  653--661. Springer, 2013.

\bibitem{odeh1964partial}
F.~Odeh and J.~B. Keller.
\newblock Partial differential equations with periodic coefficients and {Bloch}
  waves in crystals.
\newblock {\em Journal of Mathematical Physics}, 5(11):1499--1504, 1964.

\bibitem{ngsolve}
J.~e.~a. Sch\"oberl.
\newblock {Netgen/NGSolve}.
\newblock \url{https://ngsolve.org/}.
\newblock Accessed: 2025-07-09.

\bibitem{stolk2013rapidly}
C.~C. Stolk.
\newblock A rapidly converging domain decomposition method for the helmholtz
  equation.
\newblock {\em Journal of Computational Physics}, 241:240--252, 2013.

\bibitem{stolk2016dispersion}
C.~C. Stolk.
\newblock A dispersion minimizing scheme for the 3-d {Helmholtz} equation based
  on ray theory.
\newblock {\em Journal of computational Physics}, 314:618--646, 2016.

\bibitem{stolk2017improved}
C.~C. Stolk.
\newblock An improved sweeping domain decomposition preconditioner for the
  {Helmholtz} equation.
\newblock {\em Advances in Computational Mathematics}, 43(1):45--76, 2017.

\bibitem{stolk2021time}
C.~C. Stolk.
\newblock A time-domain preconditioner for the helmholtz equation.
\newblock {\em SIAM Journal on Scientific Computing}, 43(5):A3469--A3502, 2021.

\bibitem{StolkEtAl2014}
C.~C. Stolk, M.~Ahmed, and S.~K. Bhowmik.
\newblock A multigrid method for the helmholtz equation with optimized coarse
  grid corrections.
\newblock {\em SIAM Journal on Scientific Computing}, 36(6):A2819--A2841, 2014.

\bibitem{stuben1982multigrid}
K.~St{\"u}ben and U.~Trottenberg.
\newblock Multigrid methods: Fundamental algorithms, model problem analysis and
  applications.
\newblock In W.~Hackbusch and U.~Trottenberg, editors, {\em Multigrid Methods:
  Proceedings of the Conference Held at K{\"o}ln-Porz, November 23--27, 1981},
  pages 1--176. Springer, 1982.

\bibitem{TrottenbergOosterleeSchueller2001}
U.~Trottenberg, C.~W. Oosterlee, and A.~Sch{\"u}ller.
\newblock {\em Multigrid}.
\newblock Academic Press Inc., San Diego, CA, 2001.
\newblock With contributions by A. Brandt, P. Oswald and K. St{\"u}ben.

\bibitem{TurkelEtAl2013}
E.~Turkel, D.~Gordon, R.~Gordon, and S.~Tsynkov.
\newblock Compact {2D} and {3D} sixth order schemes for the {Helmholtz}
  equation with variable wave number.
\newblock {\em Journal of Computational Physics}, 232(1):272 -- 287, 2013.

\bibitem{vion2014double}
A.~Vion and C.~Geuzaine.
\newblock Double sweep preconditioner for optimized schwarz methods applied to
  the helmholtz problem.
\newblock {\em Journal of Computational Physics}, 266:171--190, 2014.

\bibitem{yovel2024lfa}
R.~Yovel and E.~Treister.
\newblock Lfa-tuned matrix-free multigrid method for the elastic helmholtz
  equation.
\newblock {\em SIAM Journal on Scientific Computing}, pages S1--S21, 2024.

\end{thebibliography}

\appendix
\section{Computation of dispersion errors}

This section contains a brief description of the numerical computation of dispersion
errors and some issues that have to be addressed. A fully detailed description
is outside the scope of this work.

In the continuous case, dispersion is about the wave vectors of
`propagating' waves, plane wave solutions
$u = e^{i x \cdot \xi}$ to the homogeneous
Helmholtz equation $-\Delta u - k^2 u =0$. In the discrete case it is
about Bloch wave solutions to the finite-element equations for the
discretized homogeneous Helmholtz equation, which exist for some $\theta$ if
\begin{equation} \label{eq:det_symbol_is_zero}
  \det \widehat{A}(\theta) = 0 .
\end{equation}

There are a couple of (conceptual) steps to determine dispersion
errors. First one finds the solutions to
(\ref{eq:det_symbol_is_zero}). The second step is to find the physical
wave vectors $\xi$ associated with the solutions $\theta$. The
dispersion error for some propagating Bloch wave is then (we consider
the error without sign) $\left| \frac{\| \xi \|}{k} - 1 \right|$. The
dispersion error depends in general on the propagation direction. To
obtain a single number we take the maximum over all propagating Bloch
waves. The discretization error is determined as a function of the
discretization and the number of dofs per wavelength.

Finding the zeros of (\ref{eq:det_symbol_is_zero}) can be done e.g.\
by numerical equation solving applied to the in absolute
value smallest eigenvalue of $\widehat{A}(\theta)$.
Identifying a solution, say $\theta^*$, with a wave vector, say
$\xi^*$ is actually a non-trivial issue. The vector $\theta$ is really in a torus
$\TT_{2\pi}^d$, and a choice of $\theta$ is really an equivalence
class
\begin{equation} \label{eq:equivalence_relation}
   [ \theta ] = \{ \theta + \alpha 2 \pi \, : \, \alpha \in \ZZ^d \} .
\end{equation}
Because of the multiple degrees of freedom per mesh
cell the Nyquist principle says that on the finite-element mesh with
order $p$ elements, wave
vectors in $\left[ - \frac{p \pi}{h}, \frac{p \pi}{h} \right]^d$ can be
represented, so the possible wave vectors are
\begin{equation} \label{eq:xi_of_theta_and_alpha}
  \xi = h^{-1} ( \theta + \alpha 2 \pi) , \qquad \text{for some
    $\alpha \in \ZZ^d$ s.t.\ $\xi \in \left[ - \frac{p \pi}{h}, \frac{p \pi}{h} \right]^d$ }.
\end{equation}
In this work we will choose $\xi$ as the choice according to
(\ref{eq:xi_of_theta_and_alpha}) for which the dispersion error
$\left| \frac{\| \xi \|}{k} - 1 \right|$ is minimal.

\NewInRevision{For the finite-difference discretization of
subsection~\ref{subsec:fd_discretization}, the discrete dispersion
relation is a closed, star-shaped curve around the origin, and
relatively easy to determine. For finite elements this is frequently
not true, in particular when there are less than two mesh cells per
wavelength. Next we sketch an argument why the discrete dispersion
relation becomes geometrically more complicated when the number
of mesh cells per wavelength is less than about two.}

With the continuous Helmholtz equation one can also associate a matrix
valued Bloch-like symbol, say $\sigma_{\rm true}(\theta)$.
\NewInRevision{While the precise expression of this symbol would
  depend on the choice of basis (cf.\ (\ref{eq:basis_E_W_theta})), }
the zeros
of $\det \sigma_{\rm true}(\theta)$ must correspond to the propagating
waves, this is the set
\begin{equation} \label{eq:equivalence_class_disprel}
  \{ [ h \xi ]  \, : \, \| \xi \| = k \}
\end{equation}
where the square brackets refer to the equivalence
classes (\ref{eq:equivalence_relation}).
The equations for discrete propagating waves $\det \widehat{A}(\theta) = 0$ can be considered
a perturbation to the continuous equations $\det \sigma_{\rm
  true}(\theta) = 0$. 
The fact that (\ref{eq:equivalence_class_disprel}) involves
equivalence classes now means it can develop self intersections.  There
are two cases. If $hk < \pi$, then
$h \xi \neq h \xi' + \alpha 2 \pi$ for $\xi, \xi'$ with
$\| \xi \| = \| \xi' \| = k$ and $\alpha \neq 0$ and
$\det \sigma_{\rm true}(\theta)$ only has single zeros that are
in general stable under
perturbation and its zeroset
is topologically a circle.  The other case is when $hk > \pi$ which
corresponds to the case that there are less than two finite-element
cells per wavelength. The true dispersion relation
(\ref{eq:equivalence_class_disprel}) then has self-intersections in
the $\theta$-torus due to the equivalence relation
(\ref{eq:equivalence_relation}).  The discrete equation
$\det \widehat{A}(\theta) = 0$ is a perturbation of the
equation $\det \sigma_{\rm true}(\theta) = 0$. The self-intersections
are not stable under perturbations and the discrete dispersion
relation is no longer a perturbation of a single circle, complicating
the computations.

To estimate dispersion errors for different order FE schemes and
different numbers of degrees of freedom per wavelength, and in view of
the above, several
procedures were implemented to find the maximum dispersion error in
the sense just defined. Let $\lambda_1(\widehat{A}(\theta))$ denote
the smallest eigenvalue in absolute value of $\widehat{A}(\theta)$.
In some cases (when dispersion error was not too small), the solutions
of (\ref{eq:det_symbol_is_zero}) could be found with sufficient
accuracy (we aimed for two digits)
by simply sampling $\lambda_1(\widehat{A}(\theta))$, then
finding sign changes, and numerically zooming in on the zeros
using a Newton-like method.  For very small dispersion errors,
sampling was done only near self intersection point of the true dispersion
relation. Away from those points the curve of zeros was close to the
correct dispersion relation and could be approximated by looking for
zeros in the direction normal to the true dispersion relation. This,
in summary,
yielded the estimates in Figure~\ref{fig:dispersionErrors} for the
dispersion errors.

\end{document}